\documentclass[12pt]{article}

\usepackage{amsmath}
\usepackage{amsfonts}
\usepackage{amsthm}
\usepackage{graphicx}
\usepackage[utf8]{inputenc}
\usepackage{color}
\usepackage{multirow}
\usepackage{multicol}
\usepackage{booktabs}
\usepackage{blindtext}
\usepackage{mathrsfs}
\usepackage{amssymb}
\usepackage{tikz}
\usepackage{bm}
\usepackage{float}
\usepackage{graphics}

\usetikzlibrary{calc}
\usepackage{xcolor}
\usepackage[
    colorlinks=true,
    linkcolor=blue,      
    citecolor=green,     
    urlcolor=blue     
]{hyperref}
\newcommand{\ma}{\textit{Mathematica$^{\small{\circledR}}$}}

\date{}

\newtheorem{prop}{Proposition}
\newtheorem{teo}{Theorem}
\newtheorem{cor}{Corollary}
\newtheorem{lema}{Lemma}
\theoremstyle{definition}

\theoremstyle{remark}

\newtheorem{rem}{Remark}
\newenvironment{pf}{\begin{proof}}{\end{proof}}

\begin{document}
\title{{Discrete orthogonal polynomials related to Hahn difference operator}}
\author{J.F. Ma\~{n}as--Ma\~{n}as$^{a,b}$, J.J. Moreno--Balc\'{a}zar$^{a,b}$,\\  M.N. Rebocho$^{c,\ast}$, C. Rodr\'{\i}guez--Perales$^{a,b}$}
\maketitle
{\scriptsize
\noindent $^a$Departamento de Matem\'{a}ticas, Universidad de Almer\'{\i}a, Spain.\\
$^b$Centro de Desarrollo y Transferencia de Investigación Matemática a la Empresa (CDTIME), Spain.\\
$^c$Departamento de Matem\'atica and CMA-UBI, Universidade da Beira Interior, Covilh\~a, Portugal.\\
Email: jmm939@ual.es, balcazar@ual.es,  mneves@ubi.pt$^{*}$, crp170@ual.es \\
*Corresponding author.
}

\begin{abstract}
 This work studies general sequences of discrete orthogonal polynomials on linear lattices, including the so-called discrete Laguerre--Hahn orthogonal polynomials.  We deduce the structure relations involving the orthogonal polynomials and their associated, as well as the explicit form of the fourth--order linear  difference equation satisfied by these Laguerre--Hahn orthogonal polynomials. Furthermore, the particular cases corresponding to semiclassical and classical families of orthogonal polynomials, which lead to second--order difference equations, are also analyzed. Several examples illustrating the computation of the explicit coefficients of these difference equations for various families of orthogonal polynomials are also presented.
\end{abstract}



\allowdisplaybreaks[4]

\section{Introduction} \label{sec:1}

 The orthogonal polynomials of a discrete variable have been extensively used in various areas of science, owing to the important role they play in problems arising in Mathematical Physics. For a broader discussion of orthogonal polynomials and their applications, we refer to the books \cite{niki-sus-uv,nik-MIR}; further insights can also be found in \cite{ismail-book,koek}.

In the present paper we focus on the sequences of orthogonal polynomials related to the Hahn difference operator \cite{hahn-q},
\begin{equation}
\left(\mathscr{D}_{q, \omega}f\right)(x)=\begin{cases}\displaystyle\frac{f(qx+\omega)-f(x)}{(q-1)x+\omega}\,, \hspace{0.3cm} \textrm{ if } x \neq \omega_0,\\
f'(\omega_0)\,, \hspace{1.3cm} \textrm{ if } x=\omega_0:=\frac{\omega}{1-q}\,,
\label{eq:hahn-op}
\end{cases}
\end{equation}
where $0<q<1$, $\omega>0$ and the prime stands for the usual derivative.
Throughout the text, for  simplification matters, and when no confusion arises, we shall  simply  denote $\left(\mathscr{D}_{q, \omega}f\right)(x)$ by  $\mathscr{D}_{q, \omega}f(x).$

  A basic and fundamental defining property of $\mathscr{D}_{q, \omega}$ is that it maps a polynomial of degree $n$ to a polynomial of degree $n-1.$ Among such operators, the Hahn  difference operator is the most general one in the class of linear lattices (cf. \cite{magnus-LHsnul,magnus-snul-sc}). Indeed, it includes the forward difference operator $\Delta$ when $q\to 1^{-}$ and $\omega=1$; the Jackson $q$--difference  operator, henceforth denoted by $\mathscr{D}_q$, when $\omega\to0^+$; and the
operator  $\mathscr{D}_{1,\omega}$ \cite{mar-mejri} (therein also called {\it{Hahn operator}}).
 Furthermore, the usual derivative $\frac{d}{dx}$ is recovered when $q \to 1^-$ and $\omega\to0^+$.

 The main purpose of the present work is to study difference--differential relations for orthogonal polynomials related to the Hahn difference operator. Our aim is to provide an unified treatment of the so--called “difference--differentiation” formulas, commonly known as structure relations, and to derive the fourth--order linear difference equations satisfied by these polynomials. Several examples from different settings are presented to illustrate the latter result.

 Structure relations play a fundamental role in the theory of orthogonal polynomials, providing a powerful tool for deriving the difference--differential equations that arise in a wide range of problems in Mathematical Physics. For instance, they  appear in many problems within the theory of polynomial solutions of hypergeometric type difference--differential equations
\cite{niki-sus-uv,nik-MIR}. Higher order difference--differential equations  appear when studying  perturbations of discrete classical orthogonal polynomials, such as in  \cite{foupou-2004}--\cite{foupou-2001}.  For continuous orthogonality on the real line, we refer the reader to the survey \cite{everit} and also \cite{buendia}, which provide an overview of orthogonal polynomial sequences satisfying differential equations of the form
\begin{equation*}
\sum_{j=0}^{N} A_j(x)\frac{d^j y}{dx^j}=0,
\label{eq:eqdifgeral-int}
\end{equation*}
where the coefficients $A_j$ are polynomials. A fundamental result due to Hahn \cite{hahn-ing} states that the order of such an equation can always be reduced to one of the two minimal cases, $N=2$ or $N=4$. The case $N=2$  corresponds to the semiclassical framework, including the classical families, whereas the minimal order for the non--semiclassical Laguerre--Hahn class is $N=4$. Concerning the Laguerre--Hahn class, a recent historical review is given in \cite{2026-Rebocho}.

In the present work, our starting point lies  in the framework of Padé approximation theory. It begins with Stieltjes functions defined via the formal  asymptotic expansion around infinity as
\begin{equation*}
	S(x)=\displaystyle\sum_{n=0}^{+\infty} u_n x^{-n-1}\,,
\label{eq:expan-S}
\end{equation*}
with a suitable sequence of moments $\{u_n\}_{n \geq 0}$, and the corresponding sequences of orthogonal polynomials, say $\{P_n\}_{n \geq 0}$, as well as the sequence of associated polynomials   $\{P^{(1)}_n\}_{n \geq 0}.$
We shall focus on the Stieltjes functions satisfying difference equations with the following form
\begin{equation}
A(\overline{x}) (\mathscr{D}_{q,\omega} S)(x) = B(x)  S(x) S(\overline{x}) + C_1(x) S(x) + C_2(x) S(\overline{x}) + D(x) \,,\label{eq:defiRic-S}
\end{equation}
where  $A, B, C_1, C_2, D$ are polynomials, with $A\not\equiv 0$, and $\overline{x}$ is defined by $\overline{x}=q x + \omega$.

The study of orthogonal polynomials related to difference equations of the type (\ref{eq:defiRic-S}) is well known in the literature, mainly within the Laguerre--Hahn theory (see \cite{fou-paco,ghressi-LH}).
The case  $B \equiv 0$ in (\ref{eq:defiRic-S})  includes the $(q,\omega)$--classical orthogonal polynomials \cite{hahn-q,smaili}, 
 see also \cite{ghressi-2009,kheriji,kheriji-mar,mejri} and
 \cite[Sec. 4.2]{ormerod-etal}, for the operator  $\mathscr{D}_{q}$.

 Our first main result concerns the derivation of the structure relations satisfied by the orthogonal polynomials related to $S$, denoted by $\{P_n\}_{n\geq 0}$, together with their associated polynomials $\{P_n^{(1)}\}_{n\geq 0},$
\begin{equation}
\begin{aligned}
A(\overline{x})(\mathscr{D}_{q,\omega}P_{n+1})(x)&=L_{n}(x)P_{n+1}(x)-C_1(x)P_{n+1}(\overline{x})+\Theta_{n}(x)P_{n}(x)-B(x)P^{(1)}_{n}(\overline{x})\,,\\
A(\overline{x})(\mathscr{D}_{q,\omega}P^{(1)}_{n})(x)&=L_{n}(x)P^{(1)}_{n}(x)+C_2(x)P^{(1)}_{n}(\overline{x})
+\Theta_{n}(x)P^{(1)}_{n-1}(x)+D(x)P_{n+1}(\overline{x})\,,\label{eq:est-intro}
\end{aligned}
\end{equation}
where $L_n, \Theta_n$ are polynomials whose degrees are determined and independent of $n$, and whose coefficients may depend on $n.$

The case $B \equiv 0$ in the above equations  should be emphasized, as those kind of relations are indeed  the generalization to the discrete setting of the structure relations characterizing the semiclassical orthogonal polynomials on the real line \cite{magnus-jcam,maroni}. Similar structure relations can be found in \cite{smaili} for the discrete classical orthogonal polynomials.

The second main result of the paper is  the fourth--order difference equation satisfied by  $\{P_n\}_{n \geq 0}$, which we present here in a comprehensive and self contained approach. The  deduction of such an equation, given in Subsection \ref{sec:5.2},  relies on the structure relations above, their successive derivatives and algebraic elimination techniques, following a similar approach from  \cite{dini} and  \cite{khal-zelia} for the continuous case.
 The fourth--order difference equation has been  obtained  for many  families of perturbed classical discrete orthogonal polynomials, see, for instance,  \cite{foupou-2004}--\cite{foupou-2001}.  In the present paper, we consider the fourth--order difference equation in its full generality. In addition to a compact representation in terms of a third--order determinant, we provide explicit formulas for computing its coefficients.

The remainder of the paper is organized as follows. In Section~\ref{sec:2} we give the main background to be used in the sequel: the Hahn  difference calculus and basic topics on orthogonal polynomials. Section~\ref{sec:3} contains the first--order difference equation for the Stieltjes transforms of  semiclassical weights.

The structure relations (\ref{eq:est-intro}) are deduced in Section \ref{sec:4}. Fundamental relations for the polynomials $L_n, \Theta_n$, based on the compatibility of the structure relations written in the matrix form combined with the recurrence relations,  are deduced in Subsection \ref{sec:4.1}.
 Section~\ref{sec:5} is devoted to the deduction  of the fourth--order difference equation for the sequences of orthogonal polynomials related to Stieltjes functions satisfying (\ref{eq:defiRic-S}). We highlight the Theorem \ref{teo3},  which is a key result to prove the fourth--order difference equation. This result provides a difference--differential equation involving the orthogonal polynomials $\{P_n\}_{n \geq 0}$ related to (\ref{eq:defiRic-S}) and their associated polynomials.
The particular cases $B \equiv 0$ therein, for semiclassical and classical orthogonal polynomials,  are discussed, and the hypergeometric--type equation is   recovered in  Corollary \ref{cor-2ndorder}. The fourth--order difference equation is deduced in Theorem \ref{teo4},  while the explicit expressions of its coefficients are given in Appendix~\ref{Appendix}.  Section~\ref{s-examples} shows  examples in the three different settings related to the derivative, forward and $q$--difference operators, with the coefficients of the corresponding difference equation explicitly provided.

\section{Notations and background:  Hahn difference operator and orthogonal polynomials}\label{sec:2}

\subsection{Hahn  difference calculus}

In this subsection we will briefly recall some of the basic properties of the Hahn  difference operator and the $(q,\omega)$--integral calculus. Further details can be found, e.g., in \cite{anaby,Filipuketal,hahn-q}.

As mentioned previously, throughout the text we shall use the notation $\overline{x}=qx+\omega$, thus, when convenient, the Hahn difference operator
 can be read as
 \begin{equation*}
(\mathscr{D}_{q,\omega}f)(x)=\frac{f(\overline{x})-f(x)}{\overline{x}-x}\,.
\end{equation*}

The chain rule applies, thus (cf. \cite[Eq. 2]{Filipuketal})
\begin{equation*}
\left(\mathscr{D}_{q,\omega}f(\overline{x})\right)(x)=q \left(\mathscr{D}_{q,\omega}f\right)(\overline{x})\,.
\label{chain rule}
\end{equation*}
For our purposes, let us define
$$\overline{x}^{\;-1}:=\frac{x-\omega}{q}, \quad  q \neq 0. $$

Basic properties of Hahn difference  operators relating $\mathscr{D}_{q,\omega}$ with $\mathscr{D}_{\frac{1}{q},\frac{-\omega}{q}}$  are listed in the following lemma. The proofs will be omitted, as they follow straightforwardly from (\ref{eq:hahn-op}). We also denote $(\mathscr{D}^n_{q,\omega}f)(x)=(\mathscr{D}_{q,\omega}(\mathscr{D}^{n-1}_{q,\omega}f)(x))(x)$.

\begin{lema}
Let $q \neq 0. $
The following properties hold:
\begin{align}
\left(\mathscr{D}_{q,\omega}f\right)(\overline{x}^{\;-1})&=\left(\mathscr{D}_{\frac{1}{q},\frac{-\omega}{q}}f\right)(x),\notag\\
\left(\mathscr{D}^2_{q,\omega}f\right)(\overline{x}^{\;-1})&=\left(\mathscr{D}_{\frac{1}{q},\frac{-\omega}{q}}\mathscr{D}_{q,\omega}f\right)(x),\label{eq:D2-xbarinv}\\
\left(\mathscr{D}_{\frac{1}{q},\frac{-\omega}{q}}\mathscr{D}_{q,\omega}f\right) (x)&= q\, \left(\mathscr{D}_{q,\omega}\mathscr{D}_{\frac{1}{q},\frac{-\omega}{q}}f\right) (x). \label{eq:propcomuta}
\end{align}
\end{lema}

We shall make use of the following {\it{invariance}} type property,
\begin{equation}
f(\overline{x})=f(x)+\varsigma(x)\mathscr{D}_{q,\omega}f(x)\,, \quad \varsigma(x)=(q-1)x+\omega.\label{eq:inv}
\end{equation}
Consequently, as $\varsigma(\overline{x})=q\,\varsigma(x)$, there follows
\begin{equation*}
f(\overline{\overline{x}})=f(x)+(1+q)\,\varsigma(x)\mathscr{D}_{q,\omega}f(x)+q \,\varsigma^2(x)\mathscr{D}^2_{q,\omega}f(x)\,,
\end{equation*}
where $\overline{\overline{x}}$ is defined recursively by $\overline{\overline{x}}=q \overline{x}+\omega.$

\begin{prop}[\cite{anaby}]
 Let $f,g: I\rightarrow\mathbb{R}$ be $(q,\omega)$--differentiable, then the following product and quotient rules hold:
\begin{enumerate}
\item Product Rule:
\begin{equation*}
 \mathscr{D}_{q, \omega}(f g)(x)=g(x) \mathscr{D}_{q, \omega} f(x)+f(\overline{x})\mathscr{D}_{q, \omega} g(x).
\end{equation*}
\item Quotient Rule:
\begin{equation*}
\mathscr{D}_{q, \omega}\left(\frac{f}{g}\right)(x)=\frac{g(x)\mathscr{D}_{q, \omega}f(x)-f(x)\mathscr{D}_{q, \omega}g(x)}{g(x)g({\overline{x}})},
\end{equation*}
assuming $g(x)g(\overline{x})\neq 0.$
\end{enumerate}
\end{prop}

The inverse of the operator $\mathscr{D}_{q, \omega}$ gives rise to the $(q,\omega)$--integral, defined as follows.
For a function  $f:I\rightarrow\mathbb{R}$
in some interval  on $\mathbb{R}$, and $a,b\in I$, the $(q,\omega)$--integral of $f$ from $a$ to $b$ is given by
\begin{equation}
\label{q omega integral definition-1}
\int_{a}^bf(t)d_{q,\omega}t:=\int_{\omega_0}^b f(t)d_{q,\omega}t-\int_{\omega_0}^a f(t)d_{q,\omega}t,
\end{equation}
where
\begin{equation}
\label{q omega integral definition}
\int_{\omega_0}^xf(t)d_{q,\omega}t:=\big(x(1-q)-\omega\big)\sum_{k=0}^\infty q^k f\left(xq^k+\omega[k]_q\right), \quad x\in I,
\end{equation}
and $[k]_q:=\frac{1-q^k}{1-q},$
provided that the series converges at $x=a$ and $x=b$. In this case, $f$ is called $(q,\omega)$--integrable on $[a,b]$.
The sum on the right--hand side of~(\ref{q omega integral definition}) is known as  Jackson--N\"{o}rlund sum (see, e.g., \cite[Sec. 2]{Filipuketal}).

The $(q,\omega)$--integration by parts is given in the following statement (see \cite{anaby}).

\begin{lema} If $f,g: I\rightarrow\mathbb{R}$ are continuous at $\omega_0$, then
\begin{equation}
\label{eq:ppp}
\int_a^b f(x) \mathscr{D}_{q, \omega}g(x) d_{q, \omega} x=\left[f(x) g(x)\right]_a ^b-\int_a^b \mathscr{D}_{q, \omega} f(x) g(\overline{x}) d_{q, \omega}x, \quad a,b\in I,
\end{equation}
where we denote $[f(x)]_a^b=f(b)-f(a).$
\end{lema}

\subsection{Orthogonal polynomials}

Let $\{u_n\}_{n \geq 0}$ be a sequence  of moments (where we take, without loss of generality,  $u_0=1$),  such that
    the Hankel determinants $$H_n:=\det[u_{j+k-2}]_{j, k=1, \dots, n}\,, \quad H_0:=1\,, $$ satisfy
     \begin{equation*}H_n \neq 0,\;\; n \geq 0\,.\label{eq:hankel}\end{equation*}
In this context, we can define the Stieltjes function as a formal power series
\begin{equation}
	S(x)=\displaystyle\sum_{n=0}^{+\infty}u_n x^{-n-1}\,.
\label{eq:expan-S}
\end{equation}
 The sequence of monic orthogonal polynomials related to $S$,
$$
P_n(x)=x^n+\cdots,\qquad n\geq0,
$$
is the sequence of diagonal Pad\'e denominators of the
rational approximants to (\ref{eq:expan-S}). We denote the corresponding
approximants by $P_n^{(1)}/P_{n+1}$, $n\geq0$, which are uniquely
determined by
\begin{equation}
P_{n+1}(x)S(x)-P_n^{(1)}(x)
=\mathcal{O}(x^{-n-2}),
\qquad n\geq0,\qquad x\to\infty.
\label{eq:Pade}
\end{equation}
The numerator polynomials $P_n^{(1)}$, which satisfy
$\deg P_n^{(1)}=n$, are commonly referred to as the associated
polynomials (see, e.g., \cite{szego,VA-lecnotes}). Throughout this work,
we also take the associated polynomials to be monic, that is,
$$
P_n^{(1)}(x)=x^n+\cdots,\qquad n\geq0.
$$

Equivalently, given the moments $\{u_n\}_{n \geq 0}$ under the above stated conditions, a linear functional $L$ can be defined on the space of polynomials by
$$u_n=L[x^n]\,, \quad n \geq 0\,.$$ Then, any solution of the Padé approximation (\ref{eq:Pade}) satisfies the orthogonality condition
\begin{equation}L[P_n(x)P_m(x)]=h_n \delta_{n,m}\,, \quad n, m = 1, 2,  \label{eq:ort}\end{equation}
that is, $\{P_n\}_{n \geq 0}$ is a sequence of orthogonal polynomials with respect to the linear functional~$L$ (see \cite{chihara,szego}).

Linear functionals related to discrete orthogonal polynomials are usually specified by a sum with a certain weight over a discrete set of lattice points. Tables with the most common weights in the literature can be collected by looking at the monographs \cite{ismail-book,koek}. In the case under consideration, whenever there exists  integral representations in terms of some weight, say $\varrho$, supported on some lattice of points $I$, the linear functional $L$ is defined via the $(q,\omega)$--integral given by (\ref{q omega integral definition-1}) and (\ref{q omega integral definition}),
\begin{equation*}
L[x^n]=\int_Ix^n \varrho(x)d_{q,\omega}x\,.
\end{equation*}
In such a case, the orthogonality condition (\ref{eq:ort}) is
$$\int_I P_n(x)P_m(x)\varrho(x)d_{q,\omega}x=h_n \delta_{n,m}\,, \quad n, m = 1, 2, \dots\, ,
$$
and $S$ given by~(\ref{eq:expan-S}) is the Stiletjes transform of $\varrho$,
\begin{equation}S(x)=\displaystyle\int_{I}\frac{\varrho(y)}{x-y}d_{q,\omega}y\,.\label{eq:S-transform}
\end{equation}

Sequences of monic orthogonal polynomials (SMOP), denoted as $\{P_n\}_{n \geq 0}$, satisfy  a three--term recurrence
relation given by \cite{szego}
\begin{equation}
P_{n+1}(x)=(x-\beta_n)P_n(x)-\gamma_n P_{n-1}(x)\,,\quad
n=0,1,2,...\,, \label{eq:ttrr-Pn}
\end{equation} with $P_{-1}(x)=0, \;
P_0(x)=1,$ and $\gamma_n\neq 0, \; n \geq 1, \; \gamma_0=u_0=1$.

The sequence  $\{P_n^{(1)}\}_{n\geq 0}$
satisfies the three--term recurrence relation
\begin{equation}
P^{(1)}_{n+1}(x)=(x-\beta_{n+1})P^{(1)}_{n}(x)-\gamma_{n+1}P^{(1)}_{n-1}(x)\,,
\quad n=0,1,2,...\,, \label{eq:ttrr-Pn1}
\end{equation} with  $P^{(1)}_{-1}(x)=0, \;P^{(1)}_0(x)=1$.

 We remark that the sequence of associated polynomials  $\{P_n^{(1)}\}_{n\geq 0}$ is, itself, a sequence of orthogonal polynomials (the result follows, e.g.,  by Favard's Theorem \cite[Th. 4.4]{chihara}). The  Stieltjes function of $\{P_n^{(1)}\}_{n\geq 0}$, henceforth denoted by $S_1,$ and the Stieltjes function $S$ of $\{P_n\}_{n \geq 0}$ are related through the following formula \cite{stieltjes}:
 \begin{equation*}
S(x)=\frac{1}{x-\beta_0-\gamma_1 S_1(x)}\,.
\end{equation*}

Furthermore, if $S$ satisfies an equation of type (\ref{eq:defiRic-S}), then standard computations using the previous relation yield that $S_1$ also satisfies an equation of type~(\ref{eq:defiRic-S}). In particular, if we denote by $A^{(1)}$, $B^{(1)}$, $C_1^{(1)}$, $C_2^{(1)}$, $D^{(1)}$ the coefficients of the equation for $S_1$, then, the following relations with the coefficients for $S$ hold
\begin{align}
A^{(1)}(x)&=A(x),\quad B^{(1)}(x)=\gamma_1 D(x),\label{A1rel} \\ C^{(1)}_1(x)&=-C_2(x)-(\overline{x}-\beta_0)D(x),\quad  C^{(1)}_2(x)=-C_1(x)-(x-\beta_0)D(x),\\
D^{(1)}(x)&=\frac{1}{\gamma_1}\big(A(\overline{x})+B(x)+(\overline{x}-\beta_0)C_1(x)+(x-\beta_0)C_2(x)+(x-\beta_0)(\overline{x}-\beta_0)D(x)\big).\label{D1rel}
\end{align}

For future purposes, we introduce the following  matrix $\mathcal{P}_n$:
\begin{equation}
\mathcal{P}_n=\left[
\begin{matrix}
  P_{n+1} & P_n^{(1)} \\
  P_n & P_{n-1}^{(1)} \\
\end{matrix}
\right], \quad n\geq 0\,.\label{eq:calP}
\end{equation}
Gathering (\ref{eq:ttrr-Pn}) and (\ref{eq:ttrr-Pn1}), $\mathcal{P}_n$ satisfies the following relation
\begin{equation}\mathcal{P}_n=\mathcal{M}_n\mathcal{P}_{n-1}\,,\;\;\;
\mathcal{M}_n=\left[
\begin{array}{cc}
  x-\beta_n & -\gamma_n \\
  1 & 0 \\
\end{array}
\right]\,,\quad n\geq 1\,. \label{eq:rrcalP}\end{equation}
 Taking determinants to (\ref{eq:rrcalP}) and iterating  we obtain
\begin{equation}
P_n^{(1)}(x)P_n(x)-P_{n+1}(x)P_{n-1}^{(1)}(x)=\prod_{k=1}^{n}\gamma_k\,, \quad n\geq 1\,.\label{eq:liou}
\end{equation}

\section{Semiclassical discrete weights and Stieltjes functions}
\label{sec:3}

In this section  we will  focus on the $(q,\omega)$--semiclassical weights  (cf. \cite{smaili}), defined via a Pearson equation
\begin{equation*}
\mathscr{D}_{q,\omega}(A \varrho)(x)=\Psi(x)\varrho(x),
\end{equation*}
or, equivalently, in the log--derivative format,
\begin{equation}
A(\overline{x})(\mathscr{D}_{q,\omega}\varrho)(x)=C(x)\varrho(x)\,, \quad C(x)=\Psi(x)-(\mathscr{D}_{q,\omega}A)(x),\label{eq:log}
\end{equation}
 where $A(\overline{x})$ and $C(x)$ are irreducible polynomials. An alternative form is \cite[Eq. (16)]{Filipuketal} (recalling the notation $\overline{x}=qx+\omega$)
 $$\mathscr{D}_{q,\omega}\varrho(x)=-u(\overline{x})\varrho(\overline{x})\,,$$ which, by virtue of (\ref{eq:inv}),  is equivalent to (\ref{eq:log}) under the correspondence $A(\overline{x})=1+\varsigma(x) u(\overline{x}), $ $ C(x)=-u(\overline{x}).$

We will   use the following definition: given a weight $\varrho$ with moments $({\varrho}_n)_{n \geq 0},$ and  a polynomial $P(x)=\sum_{k=0}^{\deg(P)}p_kx^k$, the polynomial $\pi(P;x):=\pi(P,\varrho;x)$ is  given by
\begin{equation}
\pi(P;x)=\sum_{k=0}^{\deg(P)-1}\left(\sum_{j=k+1}^{\deg(P)} p_j \varrho_{j-k-1}\right)x^k\,. \label{eq:pol-aux}
\end{equation}

\begin{lema}\label{lemma:ric-sc}
Let $\varrho$ be a semiclassical weight satisfying the Pearson equation \begin{equation}
\mathscr{D}_{q,\omega}(A \varrho)(x)=\Psi{(x)}\varrho{(x)}\,,\label{eq:Pearson}
\end{equation}
with $\lim_{x \to a, b}A(x) \varrho(x)=0$ at the endpoints of the support of $\varrho$.
Then, the Stieltjes transform of the weight $\varrho$, defined by (\ref{eq:S-transform}),
satisfies the difference equation
\begin{equation}
A(\overline{x}) (\mathscr{D}_{q,\omega} S)(x) = C(x) S(x) +  D(x)\,,\label{eq:ric-S-sc}
\end{equation}
where
\begin{equation*}
C(x)=\Psi(x)-(\mathscr{D}_{q,\omega}A)(x)\,, \quad D(x)=(\mathscr{D}_{q,\omega}\pi(A;x))(x)-\pi(\Psi;x)\,, 
\end{equation*} and the polynomial $\pi (\cdot; x)$ is defined in accordance with (\ref{eq:pol-aux}) taking the moments of the  weight~$\varrho$.
\end{lema}
\begin{pf}
Following \cite[pg. 220]{magnus-jcam}, let us firstly note the following reproducing property for $S$,
\begin{equation}
A(x)S(x)=\displaystyle \int_{I}\frac{A(y)}{x-y}\varrho(y)d_{q,\omega}y+\pi(A; x)\,,\label{eq:rep-prop}
\end{equation}
where $\pi(A; x)$ is the polynomial defined by $\pi(A; x)=\int_{I}\frac{(A(x)-A(y))}{x-y}\varrho(y)d_{q,\omega}y\,,$
or, equivalently, in terms of  the moments $\varrho_n$  of  $\varrho$,  as
\begin{equation*}
\pi(A;x)=\sum_{k=0}^{\deg(A)-1}\left(\sum_{j=k+1}^{\deg(A)}a_j {\varrho}_{j-k-1}\right)x^k\,,
\end{equation*}
where $A(x)=\displaystyle \sum_{k=0}^{\deg(A)}a_k x^k$.

Let us now adopt, whenever convenient, the notation $\mathscr{D}^{[\cdot]}_{q,\omega}$ to denote the derivative with respect to the variable $\cdot$ in a given expression.

By deriving (\ref{eq:rep-prop}) we get
\begin{equation}\mathscr{D}^{[x]}_{q,\omega} \left(A(x)S(x)\right)=\int_{I}\mathscr{D}^{[x]}_{q,\omega}\left(\frac{1}{x-y}\right)A(y)\varrho(y)d_{q,\omega}y+\mathscr{D}^{[x]}_{q,\omega}\pi(A;x)\,.\label{eq:deriv1}
\end{equation}

A straightforward computation shows that $$\mathscr{D}^{[x]}_{q,\omega}\left(\frac{1}{x-y}\right)=-\mathscr{D}^{[y]}_{q,\omega}\left(\frac{1}{x-\overline{y}^{\;-1}}\right)\,,$$
then we have
\begin{equation}
\mathscr{D}^{[x]}_{q,\omega} \left(A(x) S(x)\right)
=-\int_{I}\mathscr{D}^{[y]}_{q,\omega}\left(\frac{1}{x-\overline{y}^{\;-1}}\right)A(y)\varrho(y)d_{q,\omega}y
+\mathscr{D}^{[x]}_{q,\omega}\pi(A;x)\,.\label{eq:deriv1}
\end{equation}

Let us now look at the first term in the right--hand side of the previous equation.  Using  integration by parts (\ref{eq:ppp}) and the Pearson equation  (\ref{eq:Pearson})  we obtain
\begin{align}
-\int_{I} \mathscr{D}^{[y]}_{q,\omega}\left(\frac{1}{x-\overline{y}^{\;-1}}\right)A(y)\varrho(y)d_{q,\omega}y=
\int_{I} \frac{1}{x-y}\mathscr{D}^{[y]}_{q,\omega}\left(A(y)\varrho(y)\right)d_{q,\omega}y
=\int_{I} \frac{1}{x-y}\Psi(y)\varrho(y)d_{q,\omega}y. \label{eq:useppp}
\end{align}

Also, by virtue of the reproducing property, we have
\begin{equation}
\int_{I} \frac{1}{x-y}\Psi(y)\varrho(y)d_{q,\omega}y=\Psi(x)S(x)-\pi(\Psi;x)\,. \label{eq:inv1}
\end{equation}
Therefore, from (\ref{eq:deriv1}), (\ref{eq:useppp}) and (\ref{eq:inv1}), we have, now dropping the upper script $[x]$,
$$\mathscr{D}_{q,\omega} \left(A(x)S(x)\right) = \Psi(x) S(x) + \mathscr{D}_{q,\omega}\pi(A;x) -\pi(\Psi ;x)\,,$$
hence the required equations follow.
\end{pf}

\section{Structure relations for discrete orthogonal polynomials}\label{sec:4}

In this section we deduce difference--differential relations, usually called  structure relations, satisfied by the sequences of orthogonal polynomials related to equations of  type (\ref{eq:defiRic-S}).
This is a key result for the forthcoming sections. The technique used in the theorem that follows can be traced back to the work of Laguerre \cite{lag}, it has been revisited many times in the literature  (see, e.g. \cite{magnus-jcam,ormerod-etal}).

\begin{teo} \label{teo2}
Let $S$ be a Stieltjes function satisfying the equation (\ref{eq:defiRic-S}), i.e.,
\begin{equation*}
A(\overline{x}) (\mathscr{D}_{q,\omega} S)(x) = B(x)  S(x) S(\overline{x}) + C_1(x) S(x) + C_2(x) S(\overline{x}) + D(x) \,, 
\end{equation*}
where  $A, B, C_1, C_2, D$ are polynomials,  $A\not\equiv 0$.
Let $\{P_n\}_{n\geq 0}$ be the
corresponding SMOP, and let $\{P_{n}^{(1)}\}_{n\geq 0}$, be the monic sequence of associated polynomials. Then, the following difference--differential relations hold, for all $n \geq 0$:
\begin{align}
A(\overline{x})(\mathscr{D}_{q,\omega}P_{n+1})(x)&=L_{n}(x)P_{n+1}(x)-C_1(x)P_{n+1}(\overline{x})+\Theta_{n}(x)P_{n}(x)-B(x)P^{(1)}_{n}(\overline{x}),\label{eq:est-Pn}\\
A(\overline{x})(\mathscr{D}_{q,\omega}P^{(1)}_{n})(x)&=L_{n}(x)P^{(1)}_{n}(x)+C_2(x)P^{(1)}_{n}(\overline{x})+\Theta_{n}(x)P^{(1)}_{n-1}(x)+D(x)P_{n+1}(\overline{x}),\label{eq:est-Pn(1)}
\end{align}
where $ \Theta_n$ and $L_n$ are polynomials with degrees such that
\begin{align}
    \deg(\Theta_n)
    &\leq \max\{\deg(A)-2, \deg(B)-2, \deg(C_1)-1, \deg(C_2)-1\},
    \label{bound deg Thetan}\\
    \deg(L_n)
    &\leq \min\Bigl\{
        \max\{\deg(A)-1, \deg(C_1), \deg(\Theta_n)-1, \deg(B)-1\}, \notag\\
    &\quad\quad\quad\,\,\max\{\deg(A)-1, \deg(C_2), \deg(\Theta_n)-1, \deg(D)+1\}
    \Bigr\},
    \label{bound deg Ln}
\end{align}
satisfying the  initial conditions
\begin{equation}
\label{eq:ic}
\begin{aligned}
L_0(x)&=-C_2(x)-(\overline{x}-\beta_0)D(x)\,, \\
\Theta_0(x)&=A(\overline{x})+B(x)+(x-\beta_0)C_2(x)+(\overline{x}-\beta_0)C_1(x)+(x-\beta_0)(\overline{x}-\beta_0)D(x).
\end{aligned}
\end{equation}
\end{teo}
\begin{pf}
Let us denote the reminder of (\ref{eq:Pade}) by $\varepsilon_n$, thus  we write
\begin{equation}P_{n+1}(x)S(x)-P_{n}^{(1)}(x)=\varepsilon_{n+1}(x)\,, \quad n \geq 0\,. \label{eq:HP}
\end{equation} Applying (\ref{eq:HP}) in  (\ref{eq:defiRic-S})
 we obtain
\begin{align*}
&A(\overline{x})\mathscr{D}_{q,\omega}\left(\frac{\varepsilon_{n+1}}{P_{n+1}}\right)(x)+A(\overline{x})\mathscr{D}_{q,\omega}\left(\frac{ P^{(1)}_{n}}{P_{n+1}}\right)(x)\\
&=B(x)\left(\frac{\varepsilon_{n+1}(x)}{P_{n+1}(x)}+\frac{ P^{(1)}_{n}(x)}{P_{n+1}(x)}\right)\left(\frac{\varepsilon_{n+1}(\overline{x})}{P_{n+1}(\overline{x})}+
\frac{P^{(1)}_{n}(\overline{x})}{P_{n+1}(\overline{x})}\right)\\
&+C_1(x)\left(\frac{\varepsilon_{n+1}(x)}{P_{n+1}(x)}+\frac{ P^{(1)}_{n}(x)}{P_{n+1}(x)}\right)+C_2(x)\left(\frac{\varepsilon_{n+1}(\overline{x})}{P_{n+1}(\overline{x})}+\frac{ P^{(1)}_{n}(\overline{x})}{P_{n+1}(\overline{x})}\right)+D(x)\,.
\end{align*}
Thus, we have
\begin{align}
&A(\overline{x})\mathscr{D}_{q,\omega}\left(\frac{P^{(1)}_{n}}{P_{n+1}}\right)(x)-B(x)\left(\frac{P^{(1)}_{n}(x)}{P_{n+1}(x)}
\frac{P^{(1)}_{n}(\overline{x})}{P_{n+1}(\overline{x})}\right)-C_1(x) \frac{P^{(1)}_{n}(x)}{P_{n+1}(x)}-C_2(x)\frac{P^{(1)}_{n}(\overline{x})}{P_{n+1}(\overline{x})}-D(x)\notag\\
&=-A(\overline{x})\mathscr{D}_{q,\omega}\left(\frac{\varepsilon_{n+1}}{P_{n+1}}\right)(x)
+B(x) \vartheta_n(x)
+C_1(x)\frac{\varepsilon_{n+1}(x)}{P_{n+1}(x)}+C_2(x)\frac{\varepsilon_{n+1}(\overline{x})}{P_{n+1}(\overline{x})}\,, \label{eq:aux-magnus}
\end{align}
where we use the notation
$$\vartheta_n(x):=\frac{\varepsilon_{n+1}(x)}{P_{n+1}(x)}
\left(\frac{\varepsilon_{n+1}(\overline{x})}{P_{n+1}(\overline{x})}+\frac{ P^{(1)}_{n}(\overline{x})}{P_{n+1}(\overline{x})}\right)+
\frac{P^{(1)}_{n}(x)}{P_{n+1}(x)}\frac{\varepsilon_{n+1}(\overline{x})}{P_{n+1}(\overline{x})} \,.$$

After multiplying (\ref{eq:aux-magnus}) by $P_{n+1}(x)P_{n+1}(\overline{x})$ we  deduce
\begin{align*}
A&(\overline{x})\left(\mathscr{D}_{q,\omega}P_{n}^{(1)}\right)(x) P_{n+1}(x)- A(\overline{x})\left(\mathscr{D}_{q,\omega}P_{n+1}\right)(x) P^{(1)}_{n}(x)-B(x) P^{(1)}_{n}(x)P^{(1)}_{n}(\overline{x})\notag\\
&-C_1(x)P^{(1)}_{n}(x)P_{n+1}(\overline{x})-C_2(x)P^{(1)}_{n}(\overline{x})P_{n+1}(x)-{D}(x)P_{n+1}(x)P_{n+1}(\overline{x})\notag\\
	 &= \left[-A(\overline{x})\mathscr{D}_{q,\omega}\left(\frac{\varepsilon_{n+1}}{P_{n+1}}\right)(x)+B(x) \vartheta_n(x)
+C_1(x)\frac{\varepsilon_{n+1}(x)}{P_{n+1}(x)}+C_2(x)\frac{\varepsilon_{n+1}(\overline{x})}{P_{n+1}(\overline{x})}\right]P_{n+1}(x)P_{n+1}(\overline{x})\,.
\end{align*}

As the  left--hand side of the previous equation is a polynomial,  that we denote by $\check{\Theta}_n$, then we write the previous equality as
\begin{align}
A(\overline{x})\left(\mathscr{D}_{q,\omega}P_{n}^{(1)}\right)(x) P_{n+1}(x)- A(\overline{x})\left(\mathscr{D}_{q,\omega}P_{n+1}\right)(x) P^{(1)}_{n}(x)-B(x) P^{(1)}_{n}(x)P^{(1)}_{n}(\overline{x})&\notag\\
-C_1(x)P^{(1)}_{n}(x)P_{n+1}(\overline{x})-C_2(x)P^{(1)}_{n}(\overline{x})P_{n+1}(x)-{D}(x)P_{n+1}(x)P_{n+1}(\overline{x})
	 &= \check{\Theta}_n(x).
\label{eq:aux2-theo-magnus}
\end{align}
Furthermore, as
\begin{align*}
\check{\Theta}_n(x)=&-A(\overline{x}) \left(\mathscr{D}_{q,\omega}\varepsilon_{n+1}\right)(x)P_{n+1}(x)+A(\overline{x})\left(\mathscr{D}_{q,\omega}P_{n+1}\right)(x)\varepsilon_{n+1}(x)\\
&+B(x)\left[\varepsilon_{n+1}(x)(\varepsilon_{n+1}(\overline{x})+P_n^{(1)}(\overline{x}))+P_n^{(1)}({x})\varepsilon_{n+1}(\overline{x})\right]\\
&+C_1(x)\varepsilon_{n+1}(x) P_{n+1}(\overline{x})+ C_2(x)\varepsilon_{n+1}(\overline{x}) P_{n+1}({x})\,,
\end{align*}
by (\ref{eq:Pade}), the degree  of $\check{\Theta}_n$ is bounded above by $\max\{\deg(A)-2, \deg(B)-2, \deg(C_1)-1, \deg(C_2)-1\}$. 

Returning to (\ref{eq:aux2-theo-magnus}), let us now use  (\ref{eq:liou}), and replace $\check{\Theta}_n$ in (\ref{eq:aux2-theo-magnus}) by
\begin{equation*}
\displaystyle \frac{\left(P_n^{(1)}(x)P_n(x)-P_{n+1}(x)P_{n-1}^{(1)}(x)\right)}{\prod_{k=1}^{n}\gamma_k}\check{\Theta}_n(x)\,.
\end{equation*}
After rearranging and using the notation
\begin{equation*}
\label{defThetan}
\Theta_n(x):=-\check{\Theta}_n(x)\left(\prod_{k=1}^{n}\gamma_k\right)^{-1},
\end{equation*}
we  obtain
\begin{equation}
\label{eq:aux4-theo-magnus}
\begin{aligned}
\left[A(\overline{x})(\mathscr{D}_{q,\omega}P_{n}^{(1)})(x) -C_2(x)P^{(1)}_{n}(\overline{x})-{D}(x)P_{n+1}(\overline{x})-{\Theta}_n(x)P_{n-1}^{(1)}(x)\right]P_{n+1}(x)\\
=\left[A(\overline{x})(\mathscr{D}_{q,\omega}P_{n+1})(x) +B(x) P^{(1)}_{n}(\overline{x})+ C_1(x)P_{n+1}(\overline{x})-{\Theta}_n(x)P_n(x)\right] P^{(1)}_{n}(x)\,.
\end{aligned}
\end{equation}
Notice that the degree of the polynomial $\Theta_n$ is the same as that of $\check{\Theta}_n$, thus, the bound in~(\ref{bound deg Thetan}) holds.

As $P_n^{(1)}$ and $P_{n+1}$ do not have common zeros (this property follows, for instance, from (\ref{eq:liou})),
then both hand sides of (\ref{eq:aux4-theo-magnus}) must have the form $L_nP_n^{(1)}P_{n+1}$, where $L_n$ is a new auxiliary polynomial of bounded degree. Thus, we have
\begin{equation*}
\begin{aligned}
A(\overline{x})(\mathscr{D}_{q,\omega}P_{n+1})(x) +B(x) P^{(1)}_{n}(\overline{x})+ C_1(x)P_{n+1}(\overline{x})-{\Theta}_n(x)P_n(x)=L_n(x)	P_{n+1}(x)\,,\\
 A(\overline{x})(\mathscr{D}_{q,\omega}P_{n}^{(1)})(x) -C_2(x)P^{(1)}_{n}(\overline{x})-{D}(x)P_{n+1}(\overline{x})-{\Theta}_n(x)P_{n-1}^{(1)}(x)=L_n(x)P^{(1)}_{n}(x)\,,
\end{aligned}
\end{equation*}
 which proves (\ref{eq:est-Pn})--(\ref{eq:est-Pn(1)}). Taking $n=0$ in the above two equations, the initial conditions given in~(\ref{eq:ic}) are deduced. Finally, comparing the degrees of the polynomials in~(\ref{eq:est-Pn}) and~(\ref{eq:est-Pn(1)}) the following two bounds for the degree of $L_n$ hold
\begin{align*}
\deg(L_n)&\leq \max\{\deg(A)-1, \deg(C_1), \deg(\Theta_n)-1, \deg(B)-1\}, \notag\\
\deg(L_n)&\leq \max\{\deg(A)-1, \deg(C_2), \deg(\Theta_n)-1, \deg(D)+1\},
\end{align*}
and as consequence the bound  given in~(\ref{bound deg Ln}) follows.
\end{pf}

Two further structure relations, which will be fundamental in the proof of Theorem~\ref{teo4}, are obtained in the next statement.

\begin{lema}
Under the conditions and notations of Theorem \ref{teo2}, the following difference--differential relations hold, for all $n \geq 1$:
\begin{align}
A(\overline{x})(\mathscr{D}_{q,\omega}P_{n})(x)&=\left(L_{n-1}(x)+(x-\beta_n)\frac{\Theta_{n-1}(x)}{\gamma_n}\right)P_{n}(x)\nonumber\\
&-\frac{\Theta_{n-1}(x)}{\gamma_n} P_{n+1}(x) -C_1(x)P_{n}(\overline{x})-B(x)P^{(1)}_{n-1}(\overline{x})\,,\label{eq:extra-est-Pn} \vspace{0.3cm}\\
A(\overline{x})(\mathscr{D}_{q,\omega}P^{(1)}_{n-1})(x)&=\left(L_{n-1}(x)+(x-\beta_n)\frac{\Theta_{n-1}(x)}{\gamma_n}\right)P^{(1)}_{n-1}(x) \nonumber\\
&-\frac{\Theta_{n-1}(x)}{\gamma_n}P^{(1)}_{n}(x)+C_2(x)P^{(1)}_{n-1}(\overline{x})+D(x)P_{n}(\overline{x})\,.\label{eq:extra-est-Pn(1)}
\end{align}
\end{lema}
\begin{pf}
To obtain (\ref{eq:extra-est-Pn}), we write (\ref{eq:est-Pn}) to $n-1$, i.e.,
$$
A(\overline{x})(\mathscr{D}_{q,\omega}P_{n})(x)=L_{n-1}(x)P_{n}(x)-C_1(x)P_{n}(\overline{x})+\Theta_{n-1}(x)P_{n-1}(x)-B(x)P^{(1)}_{n-1}(\overline{x})\,,$$
and then use the three--term recurrence relation $$P_{n-1}(x)=\frac{(x-\beta_n)}{\gamma_n}P_n(x)-\frac{1}{\gamma_n}P_{n+1}(x)\,.$$
The equation (\ref{eq:extra-est-Pn(1)}) follows by  analogous procedure, starting with equation (\ref{eq:est-Pn(1)}).
\end{pf}

\subsection{Further relations for the polynomials $\Theta_n$, $L_n$}\label{sec:4.1}

The main purpose of this section is to obtain equations for the polynomials $\Theta_n, L_n$ defining the structure relations that were previously deduced. Firstly, in the following statement, we derive an equation for the  matrices introduced in (\ref{eq:calP}).

\begin{lema}
Let  the Stieltjes function $S$ given by (\ref{eq:expan-S}) satisfy the difference equation (\ref{eq:defiRic-S}),
$$A(\overline{x}) (\mathscr{D}_{q,\omega} S)(x) = B(x)  S(x) S(\overline{x}) + C_1(x) S(x) + C_2(x) S(\overline{x}) + D(x) \,,$$
and let $\{\mathcal{P}_n\}_{n \geq 0}$ be the corresponding sequence defined by (\ref{eq:calP}).
Then, the following matrix Sylvester equation holds:
\begin{equation}
	A(\overline{x}) (\mathscr{D}_{q,\omega}\mathcal{P}_{n})(x)=\mathcal{A}_n(x) \mathcal{P}_{n}(x)-\mathcal{P}_{n}(\overline{x})\mathcal{C}(x)\,, \quad n \geq 0\,, \label{eq:sylv-calPn}
\end{equation}
where
\begin{equation}
	\mathcal{A}_n=
	\begin{bmatrix}
		 L_{n} & \Theta_{n} \\
		-\Theta_{n-1}/\gamma_n &\;\;
		L_{n-1}+(x-\beta_n)\Theta_{n-1}/\gamma_n
	\end{bmatrix}\,,
	\;\;\mathcal{C}=
	\begin{bmatrix}
		 C_1 & -D \\
		B & -C_2
	\end{bmatrix}\,,\label{eq:calAn-C}
\end{equation}
with the initial conditions
\begin{equation}
L_{-1}(x)=C_1(x),\quad \frac{\Theta_{-1}(x)}{\gamma_0}=D(x)\,. \label{eq:ic2}
\end{equation}
\end{lema}
\begin{proof}
The result follows from definition (\ref{eq:calP}) and coupling the equations (\ref{eq:est-Pn}),  (\ref{eq:est-Pn(1)}), (\ref{eq:extra-est-Pn}) and (\ref{eq:extra-est-Pn(1)}) into a matrix compact form. Note that the result is also valid for $n=0$ taking the initial conditions given in~(\ref{eq:ic2}).
\end{proof}

\begin{cor}
The transfer matrices $\mathcal{M}_n$  defined in (\ref{eq:rrcalP}) satisfy the difference equation
\begin{equation}
	A(\overline{x})(\mathscr{D}_{q,\omega}\mathcal{M}_{n})(x)=\mathcal{A}_n(x) \mathcal{M}_{n}(x)-\mathcal{M}_{n}(\overline{x})\mathcal{A}_{n-1}(x)\,, \quad n \geq 0\,.\label{eq:sylv-transfer}
\end{equation}\end{cor}
\begin{pf}
Using (\ref{eq:rrcalP}) in (\ref{eq:sylv-calPn}) we obtain
\begin{multline*}
A(\overline{x}) (\mathscr{D}_{q,\omega}\mathcal{M}_n)(x) \mathcal{P}_{n-1}(x)+\mathcal{M}_n(\overline{x}) A(\overline{x})(\mathscr{D}_{q,\omega}\mathcal{P}_{n-1})(x)\\=\mathcal{A}_n(x) \mathcal{M}_n(x)\mathcal{P}_{n-1}(x)
-\mathcal{M}_n(\overline{x})\mathcal{P}_{n-1}(\overline{x})\mathcal{C}(x)\,.\end{multline*}
  Taking $n-1$ in (\ref{eq:sylv-calPn}), i.e., $A(\overline{x})(\mathscr{D}_{q,\omega}\mathcal{P}_{n-1})(x)=\mathcal{A}_{n-1}(x) \mathcal{P}_{n-1}(x)-\mathcal{P}_{n-1}(\overline{x})\mathcal{C}(x), $  and applying it in the above equation we obtain equation (\ref{eq:sylv-transfer}) after some straightforward simplifications.
\end{pf}

Note that the  matrix equation (\ref{eq:sylv-transfer}) encloses two  nontrivial equations  from positions $(1,1)$ and $(1,2)$, respectively, yielding
\begin{eqnarray}
A(\overline{x})&=&(x-\beta_n)L_n(x)+\Theta_n(x)-(\overline{x}-\beta_n)L_{n-1}(x)-\gamma_n\frac{\Theta_{n-2}(x)}{\gamma_{n-1}}\,,\notag\\
0&=&- L_n(x)-(\overline{x}-\beta_n)\frac{\Theta_{n-1}(x)}{\gamma_n}+L_{n-2}(x)+(x-\beta_{n-1})\frac{\Theta_{n-2}(x)}{\gamma_{n-1}}\,. \label{eq:pos1-2}
\end{eqnarray}

\begin{cor} The polynomials $L_n, \Theta_n$ in 
(\ref{eq:calAn-C}) satisfy
\begin{eqnarray}
&&L_{n-1}(x)+L_n(x)+(x-\beta_n)\frac{\Theta_{n-1}(x)}{\gamma_n}=C_1(x)-C_2(x)-\varsigma(x)\sum_{k=0}^{n}\frac{\Theta_{k-1}(x)}{\gamma_k}\,, \label{eq:trace}\\
&&L_n(x)\left(L_{n-1}(x)+(x-\beta_n)\frac{\Theta_{n-1}(x)}{\gamma_n}\right)+\Theta_n(x)\frac{\Theta_{n-1}(x)}{\gamma_n}\nonumber\\
&& \hspace{4cm}= -C_1(x)C_2(x)+B(x)D(x)+A(\overline{x})\sum_{k=0}^{n}\frac{\Theta_{k-1}(x)}{\gamma_k}\,. \label{eq:det}
\end{eqnarray}
\end{cor}
\begin{pf}
To deduce (\ref{eq:trace}), we begin by  using $$\overline{x}-\beta_n=x-\beta_n+\varsigma(x)\,,$$ in equation (\ref{eq:pos1-2}). Then, we obtain
$$L_n(x)+(x-\beta_n)\frac{\Theta_{n-1}(x)}{\gamma_n}=L_{n-2}(x)+(x-\beta_{n-1})\frac{\Theta_{n-2}(x)}{\gamma_{n-1}}-\varsigma(x)\frac{\Theta_{n-1}(x)}{\gamma_n}.$$
Thus, adding $L_{n-1}$, we have,
for all $n \geq 1,$
\begin{multline*}L_{n-1}(x)+L_n(x)+(x-\beta_n)\frac{\Theta_{n-1}(x)}{\gamma_n}\\
=L_{n-2}(x)+L_{n-1}(x)+(x-\beta_{n-1})\frac{\Theta_{n-2}(x)}{\gamma_{n-1}}-\varsigma(x)\frac{\Theta_{n-1}(x)}{\gamma_n}\,.\end{multline*}
Iterating the above relation, we deduce
\begin{multline*}
L_{n-1}(x)+L_n(x)+(x-\beta_n)\frac{\Theta_{n-1}(x)}{\gamma_n}\\
=L_{-1}(x)+L_0(x)+(x-\beta_0)\Theta_{-1}(x)-\varsigma(x)\sum_{k=1}^{n}\frac{\Theta_{k-1}(x)}{\gamma_k}.
\end{multline*}
Thus, using  the initial conditions (\ref{eq:ic}), we obtain
$$L_{n-1}(x)+L_n(x)+(x-\beta_n)\frac{\Theta_{n-1}(x)}{\gamma_n}=C_1(x)-C_2(x)-\varsigma(x)D-\varsigma(x)\sum_{k=1}^{n}\frac{\Theta_{k-1}{(x)}}{\gamma_k}\,.$$
Furthermore, as $D=\frac{\Theta_{-1}}{\gamma_0}$, then the formula above  can  written as (\ref{eq:trace}).

To obtain the formula (\ref{eq:det})  we take determinants in (\ref{eq:sylv-transfer}). Then, for all $n \geq 0,$
$$\det(\mathcal{A}_n(x) \mathcal{M}_{n}(x))=\det\left(\mathcal{M}_{n}(\overline{x})\mathcal{A}_{n-1}(x)+A(\overline{x}) (\mathscr{D}_{q,\omega}\mathcal{M}_{n})(x)\right)\,.$$
Noting that $A(\overline{x}) (\mathscr{D}_{q,\omega}\mathcal{M}_{n})(x)=\begin{bmatrix}A(\overline{x}) & 0\\
0 & 0
\end{bmatrix}$, there follows
\begin{equation}\det(\mathcal{A}_n(x))\det(\mathcal{M}_{n}(x))=\det(\mathcal{M}_{n}(\overline{x}))\det(\mathcal{A}_{n-1}(x))+A(\overline{x}) \left[\mathcal{M}_{n}(\overline{x})\mathcal{A}_{n-1}(x)\right]_{(2,2)}\,,\label{eq:aux1-det}\end{equation}
where $\left[\mathcal{M}_{n}(\overline{x})\mathcal{A}_{n-1}(x)\right]_{(2,2)}$ denotes the element in position $(2,2)$ of the matrix $\mathcal{M}_{n}(\overline{x})\mathcal{A}_{n-1}(x)$.

Taking into account that $\det(\mathcal{M}_{n})=\gamma_n$ and $\left[\mathcal{M}_{n}(\overline{x})\mathcal{A}_{n-1}(x)\right]_{(2,2)}=\Theta_{n-1}(x), $
 by (\ref{eq:aux1-det}) we obtain, for all $n \geq 1,$
\begin{equation*}
\det(\mathcal{A}_{n}(x))=\det(\mathcal{A}_{n-1}(x))+A(\overline{x}) \frac{\Theta_{n-1}(x)}{\gamma_n}=\cdots=\det(\mathcal{A}_0(x))+A(\overline{x}) \sum_{k=1}^{n}\frac{\Theta_{k-1}(x)}{\gamma_k}.\label{eq:aux2-det}
\end{equation*}
 Using the initial conditions (\ref{eq:ic}), we have
 $$\det (\mathcal{A}_n(x)) = -C_1(x)C_2(x)+(A(\overline{x})+B(x))D(x)+ A(\overline{x})\sum_{k=1}^{n}\frac{\Theta_{k-1}(x)}{\gamma_k}\,, $$
 which we can write as (\ref{eq:det}) upon the identification $D=\frac{\Theta_{-1}}{\gamma_0}$.
  \end{pf}

\begin{rem}
Equations (\ref{eq:trace}) and (\ref{eq:det}) will be used to simplify the derivation of the hypergeometric type difference equation for the $(q,\omega)$--classical orthogonal polynomials (cf. Theorem~\ref{teo4-hyper}).
\end{rem}

\section{Difference  equations for Laguerre--Hahn orthogonal polynomials}\label{sec:5}

In this section we deduce the fourth--order differential equation satisfied by the orthogonal polynomials $\{P_n\}_{n \geq 0}$ related to a Stieltjes function such that equation (\ref{eq:defi-Ric-S}) holds. Furthermore, the cases when these polynomials correspond to classical and semiclassical families of orthogonal polynomials are also addressed.

 First, we observe that equations (\ref{eq:est-Pn})--(\ref{eq:est-Pn(1)})
can be rewritten as
\begin{align}
\tilde{A}_1({x})(\mathscr{D}_{q,\omega}P_{n+1})(x)&=\left(L_{n}-C_1\right)(x)P_{n+1}(x)+\Theta_{n}(x)P_{n}(x)-B(x)P^{(1)}_{n}(\overline{x})\,,\label{eq:(11)1}
\\
\tilde{A}_2({x})(\mathscr{D}_{q,\omega}P^{(1)}_{n})(x)&=\left(L_{n}+C_2\right)(x)P^{(1)}_{n}(x)+\Theta_{n}(x)P^{(1)}_{n-1}(x)+D(x)P_{n+1}(\overline{x})\,,\label{eq:(12)1}
\end{align}
where
\begin{equation}
\label{At12}
\tilde{A}_1(x)=A(\overline{x})+\varsigma(x)C_1(x)\,,\;\; \tilde{A}_2(x)=A(\overline{x})-\varsigma(x)C_2(x)\,.
\end{equation}
 Similarly, equations (\ref{eq:extra-est-Pn}) and (\ref{eq:extra-est-Pn(1)}) can be expressed as
\begin{align}
\tilde{A}_1({x})(\mathscr{D}_{q,\omega}P_{n})(x)&=(\tilde{L}_{n-1}-C_1)(x)P_{n}(x)-\frac{\Theta_{n-1}(x)}{\gamma_n} P_{n+1}(x)-B(x)P^{(1)}_{n-1}(\overline{x})\,,\label{eq:(21)1} \\
\tilde{A}_2({x})(\mathscr{D}_{q,\omega}P^{(1)}_{n-1})(x)&=(\tilde{L}_{n-1}+C_2)(x)P^{(1)}_{n-1}(x) -\frac{\Theta_{n-1}(x)}{\gamma_n}P^{(1)}_{n}(x)+D(x)P_{n}(\overline{x})\,,\label{eq:(22)1}
\end{align}
where we use the notation
\begin{equation}
\label{Lt}
\tilde{L}_{n-1}(x):=L_{n-1}(x)+(x-\beta_n)\frac{\Theta_{n-1}(x)}{\gamma_n}\,.
\end{equation}

Taking into account the above equations, we now establish the following key preliminary result.
\begin{teo} \label{teo3}
Let $\{P_n\}_{n\geq 0}$ be a SMOP related to a Stieltjes function  $S$  satisfying
\begin{equation*}
A(\overline{x}) (\mathscr{D}_{q,\omega} S)(x) = B(x)  S(x) S(\overline{x}) + C_1(x) S(x) + C_2(x) S(\overline{x}) + D(x) \,,\label{eq:defi-Ric-S}
\end{equation*}
where  $A, B, C_1, C_2, D$ are polynomials,  $A\not \equiv 0$.
Then,  the following equation holds:
\begin{multline}
\Theta_n(x)\tilde{A}_1(x)\tilde{A}_1({\overline{x}})(\mathscr{D}^2_{q,\omega}P_{n+1})(x) +J_n(x)(\mathscr{D}_{q,\omega}P_{n+1})(x)+
K_n(x) P_{n+1}(x)\\
=-B(x)\Theta_{n}(x)\Theta_{n}(\overline{x})P^{(1)}_{n-1}(\overline{x})+\left[B(x)T_n(x)-\tilde{A}_1(x)\Theta_{n}(x)(\mathscr{D}_{q,\omega}B)(x)\right]P^{(1)}_{n}(\overline{x})\\
-q\tilde{A}_1(x)B(\overline{x})\Theta_{n}(x) \left(\mathscr{D}_{q,\omega}P^{(1)}_{n}\right)(\overline{x})\,, \label{eq:key3}
\end{multline}
 with $\tilde{A}_1(x)$ defined by~(\ref{At12}),
  \begin{align*}
 J_n(x)&=\Theta_n(x)\tilde{A}_1(x)\left[(\mathscr{D}_{q,\omega}\tilde{A}_1)({x})-(L_{n}-C_1)({\overline{x}})\right]-\tilde{A}_1(x)T_n(x),\\
K_n(x)&=\Theta_n(x)\left[-\tilde{A}_1(x)\mathscr{D}_{q,\omega}(L_{n}-C_1)(x)+  \Theta_{n}(\overline{x}) \frac{\Theta_{n-1}(x)}{\gamma_n}\right]+T_n(x)(L_n-C_1)(x)\,,
\end{align*}
and
\begin{equation*}
T_n(x)= \tilde{A}_1(x){(\mathscr{D}_{q,\omega}\Theta_n)}(x)+\Theta_{n}(\overline{x}) (\tilde{L}_{n-1}-C_1)(x)\,.
 \end{equation*}
\end{teo}

\begin{pf}
Taking derivatives in  (\ref{eq:(11)1}) we  have
\begin{multline}
(\mathscr{D}_{q,\omega}\tilde{A}_1)({x})(\mathscr{D}_{q,\omega}P_{n+1})(x)+\tilde{A}_1({\overline{x}})(\mathscr{D}^2_{q,\omega}P_{n+1})(x)\\
=\mathscr{D}_{q,\omega}\left(L_{n}-C_1\right)(x) P_{n+1}(x)+\left(L_{n}-C_1\right)({\overline{x}})(\mathscr{D}_{q,\omega}P_{n+1})(x)+
(\mathscr{D}_{q,\omega}\Theta_{n})(x)P_{n}(x)\\+\Theta_{n}(\overline{x})(\mathscr{D}_{q,\omega}P_n)(x)
-(\mathscr{D}_{q,\omega}B)(x) P^{(1)}_{n}(\overline{x})-B(\overline{x})(\mathscr{D}_{q,\omega}P^{(1)}_{n}(\overline{x}))(x)\,.\label{eq:1}
\end{multline}
Multiplying (\ref{eq:1}) by $\tilde{A}_1(x)$ and using (\ref{eq:(21)1}) in the resulting equation then we  eliminate $(\mathscr{D}_{q,\omega}P_n)(x),$  and obtain
\begin{align*}
&\tilde{A}_1(x)\tilde{A}_1({\overline{x}})(\mathscr{D}^2_{q,\omega}P_{n+1})(x)
+\tilde{A}_1(x)\left[(\mathscr{D}_{q,\omega}\tilde{A}_1)(x)-(L_{n}-C_1)({\overline{x}})\right](\mathscr{D}_{q,\omega}P_{n+1})(x)\\
&+\left[-\tilde{A}_1(x)\mathscr{D}_{q,\omega}(L_{n}-C_1)(x)+  \Theta_{n}(\overline{x}) \frac{\Theta_{n-1}(x)}{\gamma_n}\right]P_{n+1}(x)\\
&+\left[ -(\mathscr{D}_{q,\omega}\Theta_{n})(x)\tilde{A}_1(x)-\Theta_{n}(\overline{x})(\tilde{L}_{n-1}-C_1)(x) \right]P_{n}(x)\\
&=-B(x)\Theta_{n}(\overline{x})P^{(1)}_{n-1}(\overline{x})-\tilde{A}_1(x)(\mathscr{D}_{q,\omega}B)(x)P^{(1)}_{n}(\overline{x})
-\tilde{A}_1(x)B(\overline{x})q \left(\mathscr{D}_{q,\omega}P^{(1)}_{n}\right)(\overline{x})\,.
\end{align*}
 Finally, we multiply the above equation by $\Theta_n(x)$ and use (\ref{eq:(11)1}), i.e.,
\begin{equation*}
\Theta_{n}(x)P_{n}(x)=\tilde{A}_1({x})(\mathscr{D}_{q,\omega}P_{n+1})(x)-\left(L_{n}-C_1\right)(x)P_{n+1}(x)+B(x)P^{(1)}_{n}(\overline{x})\,,
\end{equation*}
to eliminate the term $P_n$, which leads to equation (\ref{eq:key3}).
\end{pf}

\subsection{The semiclassical setting and the  hypergeometric type difference equations}

In what follows, we establish the second--order difference equation for semiclassical and classical orthogonal polynomials. This result is obtained applying Theorem \ref{teo3} by setting $B \equiv 0$ and $C_2\equiv 0$ therein.
For ease of exposition, we slightly modify the notation throughout this subsection, that is, we take $C_1:=C$ and $\tilde{A}(x):=\tilde{A}_1(x).$ Therefore,
\begin{equation}
\label{At}
\tilde{A}(x):=A(\overline{x})+\varsigma(x)C(x)\,.
\end{equation}
\begin{cor}  \label{cor-2ndorder}
Let $\{P_n\}_{n\geq 0}$ be a sequence of monic semiclassical orthogonal polynomials related to a Stieltjes function  $S$  satisfying
\begin{equation}
A(\overline{x})( \mathscr{D}_{q,\omega} S)(x) =  C(x) S(x)  + D(x) \,,\label{eq:sc-Ric-S}
\end{equation}
where $A, C, D$ are polynomials.
 Then,  $\{P_n\}_{n\geq 0}$  satisfies the second--order difference equation
\begin{equation}
\widehat{A}_n(x)(\mathscr{D}^{2}_{q,\omega}P_{n+1})(x)+\widehat{B}_n(x)(\mathscr{D}_{q,\omega}P_{n+1})(x)+\widehat{C}_n(x)P_{n+1}(x)=0\,, \label{eq:2ndorder-sc}
\end{equation}
where
  \begin{eqnarray}
  \widehat{A}_n(x)&=&\Theta_n(x)\tilde{A}(x)\tilde{A}(\overline{x})\,,\label{eq:hatAn}\\
 \widehat{B}_n(x)&=&\Theta_n(x)\tilde{A}(x)\left[(\mathscr{D}_{q,\omega}\tilde{A})({x})-(L_{n}-C)({\overline{x}})\right]\nonumber\\
&& \hspace{1cm}- \tilde{A}(x)\left[\tilde{A}(x){(\mathscr{D}_{q,\omega}\Theta_n)}(x) + \Theta_{n}(\overline{x})(\tilde{L}_{n-1}-C)(x)\right]\,,\label{eq:hatBn}\\
\widehat{C}_n(x)&=&\Theta_n(x)\left[-\tilde{A}(x)\mathscr{D}_{q,\omega}(L_{n}-C)(x)+  \Theta_{n}(\overline{x}) \frac{\Theta_{n-1}(x)}{\gamma_n}\right]\nonumber\\
&&+\left[\tilde{A}(x)(\mathscr{D}_{q,\omega}\Theta_n)(x)+\Theta_{n}(\overline{x}) (\tilde{L}_{n-1}-C)(x)\right](L_n-C)(x)\,, \label{eq:hatCn}
\end{eqnarray}
and $\tilde{A}(x)$ is given by~(\ref{At}).
\end{cor}
\begin{pf}
 It is enough to take $B \equiv 0$ in Theorem \ref{teo3}.
 \end{pf}

Furthermore, for classical families of orthogonal polynomials, the  restrictions on the degrees of the coefficients of the difference equation (\ref{eq:sc-Ric-S})  allow to simplify equation (\ref{eq:2ndorder-sc}), as shown in the following corollary.

\begin{cor}  \label{cor-2ndorder-classical}
Let $\{P_n\}_{n\geq 0}$ be a a SMOP related to a Stieltjes function  $S$  satisfying equation (\ref{eq:sc-Ric-S}) with the following restrictions on the coefficients:
$$\deg(A) \leq 2, \quad\deg(C)=1\,.$$
 Then, the classical polynomials $\{P_n\}_{n\geq 0}$  satisfy the second--order difference equation
\begin{equation}
\tilde{A}(\overline{x})(\mathscr{D}^{2}_{q,\omega}P_{n+1})(x)+\left[(\mathscr{D}_{q,\omega}\tilde{A})({x})+C(x)+
\varsigma(x)b_n\right]
(\mathscr{D}_{q,\omega}P_{n+1})(x)+{b}_n P_{n+1}(x)=0\,,\label{eq:raw-hyper}
\end{equation}
where the  constants ${b}_n$ are given by
  \begin{equation}
{b}_{n}=\sum_{k=0}^{n}\frac{\Theta_{k-1}}{\gamma_k}-\mathscr{D}_{q,\omega}(L_n-C)(x)\,.\label{eq:tildebn}
  \end{equation}
\end{cor}
\begin{pf}The degree restrictions $\deg(A) \leq 2, \;\deg(C)=1$, imply that $\deg(\Theta_n)=0$.  Hence, $\Theta_n$ is a constant and we can omit its dependence on $x$ throughout this proof. Consequently, $\Theta_n$ can be removed from the coefficients $\widehat{A}_n$, $\widehat{B}_n$, $\widehat{C}_n$ given by (\ref{eq:hatAn})--(\ref{eq:hatCn}), allowing these coefficients to be further simplified.

First, we simplify $\widehat{B}_n$. From (\ref{eq:hatBn}), it holds
$$\widehat{B}_n(x)=\tilde{A}(x)\left[(\mathscr{D}_{q,\omega}\tilde{A})({x})-(L_{n}-C)({\overline{x}})-\tilde{L}_{n-1}(x)+C(x)\right]\,,$$
and using the invariance   property (\ref{eq:inv}),  we  have that
\begin{equation*}
\widehat{B}_n(x)=\tilde{A}(x)\left[(\mathscr{D}_{q,\omega}\tilde{A})({x})-(L_{n}(x)+\tilde{L}_{n-1}(x))+2C(x)
-\varsigma(x)\mathscr{D}_{q,\omega}(L_n-C)({x}) \right].
\end{equation*}
 On the other hand, by (\ref{eq:trace}) and~(\ref{Lt}) we deduce that
\begin{equation}
L_{n}(x)+\tilde{L}_{n-1}(x)=C(x)-\varsigma(x)\sum_{k=0}^{n}\frac{\Theta_{k-1}}{\gamma_k},\label{sumLn}
\end{equation}
 so $\widehat{B}_n$ can be simplified as
\begin{align*}
\widehat{B}_n(x)&=\tilde{A}(x)\left[(\mathscr{D}_{q,\omega}\tilde{A})({x})+C(x)+\varsigma(x)b_n \right],
\end{align*}
 where $b_n$ is defined by~(\ref{eq:tildebn}). We note that $b_n$  are constants, independent of $x,$ since \mbox{$\deg(L_n)\leq1$}, as follows from~(\ref{bound deg Ln}).

Let us now simplify $\widehat{C}_n$. Using (\ref{eq:hatCn}), we obtain
$$\widehat{C}_n(x)=-\tilde{A}(x)\mathscr{D}_{q,\omega}(L_{n}-C)(x)+  \Theta_{n} \frac{\Theta_{n-1}}{\gamma_n}+(L_n-C)(x)(\tilde{L}_{n-1}-C)(x)\,.
$$
Next, from~(\ref{eq:ic2}),~(\ref{eq:det}),~(\ref{Lt}), we  derive
$$  \Theta_{n} \frac{\Theta_{n-1}}{\gamma_n}+L_n(x)\tilde{L}_{n-1}(x)=A(\overline{x})D(x)+A(\overline{x})\sum_{k=1}^{n}\frac{\Theta_{k-1}}{\gamma_k}\,,$$
and taking into account~(\ref{sumLn}),  $\widehat{C}_n$ simplifies as
\begin{align*}
 \widehat{C}_n(x)= &-\tilde{A}(x)\mathscr{D}_{q,\omega}(L_{n}-C)(x)+A(\overline{x})D(x)+A(\overline{x})\sum_{k=1}^{n}\frac{\Theta_{k-1}}{\gamma_k}\\
 &-C(x)
 \left(C(x)-\varsigma(x)\sum_{k=0}^{n}\frac{\Theta_{k-1}}{\gamma_k}\right)+C^2(x).
  \end{align*}
 After some computations,
we obtain
\begin{align*}
\widehat{C}_n(x)= \tilde{A}(x)b_n.
\end{align*}
Thus, we obtain the following second--order difference equation
\begin{equation*}
\tilde{A}({x})\left[\tilde{A}(\overline{x})(\mathscr{D}^{2}_{q,\omega}P_{n+1})(x)+\left[(\mathscr{D}_{q,\omega}\tilde{A})({x})
+C(x)+\varsigma(x)b_n\right]
(\mathscr{D}_{q,\omega}P_{n+1})(x)+{b}_n P_{n+1}(x)\right]=0.
\end{equation*}
Finally, since $\deg\tilde{A}({x})$ is fixed,
 we obtain  equation~(\ref{eq:raw-hyper}).
\end{pf}

The following  lemma will be used to show further  simplifications, namely, the hypergeometric form of equation (\ref{eq:raw-hyper}) to be deduced in Theorem \ref{teo4-hyper}.
\begin{lema}
Let $A,\,C $ be the polynomials given in equation~(\ref{eq:sc-Ric-S}), $\tilde{A}$ the polynomial defined in~(\ref{At}), and $\Psi(x)=(\mathscr{D}_{q,\omega}{A})(x)+C(x)$. Then, the following identities hold:
\begin{gather}
(\mathscr{D}_{q,\omega}\tilde{A})({x})+C(x)=q\Psi(\overline{x})\,, \label{eq:simp-aux-pearson}\\
\tilde{A}(\overline{x})=A(\overline{x})+q\varsigma(x)\Psi(\overline{x})\,. \label{eq:simp-aux-tildeA}
\end{gather}
\end{lema}
\begin{pf}
Using~(\ref{eq:inv}), we obtain
\begin{eqnarray*}
(\mathscr{D}_{q,\omega}\tilde{A})({x})+C(x)=q\left(\mathscr{D}_{q,\omega}A\right)(\overline{x})+\left(\mathscr{D}_{q,\omega}\varsigma\right)(x)C(x)
+\varsigma(\overline{x})(\mathscr{D}_{q,\omega}C)(x)  \,+\, C(x)\,.
\end{eqnarray*}
Taking into account that $(\mathscr{D}_{q,\omega}\varsigma)(x)=q-1,$ $\varsigma(\overline{x})=q\varsigma(x)$ and using~(\ref{eq:inv}) once more, the following equation holds,
$$(\mathscr{D}_{q,\omega}\tilde{A})({x})+C(x)=q \left(\mathscr{D}_{q,\omega}{A}+C\right)(\overline{x})\,,$$
which gives us the identity (\ref{eq:simp-aux-pearson}).

To deduce (\ref{eq:simp-aux-tildeA}) we firstly note that, by virtue of the invariance  property (\ref{eq:inv}), we have $\tilde{A}(\overline{x})=\tilde{A}(x)+\varsigma(x)(\mathscr{D}_{q,\omega}\tilde{A})(x)$.
Thus, using (\ref{eq:simp-aux-pearson}), we  obtain (\ref{eq:simp-aux-tildeA}). 
\end{pf}

\begin{teo}\label{teo4-hyper}
Let $\varrho$ be a semiclassical weight satisfying the Pearson equation (\ref{eq:Pearson}), \begin{equation*}
\mathscr{D}_{q,\omega}(A \varrho)(x)=\Psi(x)\varrho(x)\,,
\end{equation*}  with
$\lim_{x \to a, b}A(x) \varrho(x)=0$ at the endpoints of the support of $\varrho$. Let the degrees of $A, \Psi$ be such that
 $\deg(A) \leq 2, \;\deg(\Psi)=1\,.$ Then, the following  hypergeometric type difference equation holds, for all $n \geq 0$:
  \begin{equation*}
 A({x})\left(\mathscr{D}_{q,\omega}\mathscr{D}_{\frac{1}{q},-\frac{\omega}{q}} P_{n+1}\right)(x)
+\Psi({x})(\mathscr{D}_{q,\omega}P_{n+1})({x})+\lambda_{n+1}  P_{n+1}({x})=0\,,\label{eq:hyper}
\end{equation*}
 with  $\lambda_{n+1}=\displaystyle \frac{b_n}{q},$ where the constants $b_n$ are given by~(\ref{eq:tildebn}).
  \end{teo}
\begin{pf}
Under the stated conditions, the  Stieltjes transform of the weight $\varrho$ satisfies (\ref{eq:ric-S-sc}) with $C(x)=\Psi(x)-(\mathscr{D}_{q,\omega}A)(x)$ and $D$, now constant,  given by $D=(\mathscr{D}_{q,\omega}\pi(A;x))(x)-\pi(\Psi;x).$
Thus, equation (\ref{eq:raw-hyper}) holds, from where applying (\ref{eq:simp-aux-pearson}) and (\ref{eq:simp-aux-tildeA}) we have
\begin{equation*}
\left(A(\overline{x})+q\varsigma(x)\Psi(\overline{x})\right)(\mathscr{D}^{2}_{q,\omega}P_{n+1})(x)\\
+\left(q\Psi(\overline{x})+\varsigma(x)b_n\right)(\mathscr{D}_{q,\omega}P_{n+1})(x)
+b_nP_{n+1}(x)=0\,.
\end{equation*}
Using the invariance property (\ref{eq:inv}) for the first and second derivatives  of $P_{n+1}$, we obtain, after  some simplifications,
\begin{equation}
A(\overline{x})(\mathscr{D}^{2}_{q,\omega}P_{n+1})(x)
+q\Psi(\overline{x})(\mathscr{D}_{q,\omega}P_{n+1})(\overline{x})+b_n P_{n+1}(\overline{x})=0\,.\label{eq:quase-hyper}
\end{equation}
 Applying properties (\ref{eq:D2-xbarinv}) and (\ref{eq:propcomuta})  to equation (\ref{eq:quase-hyper}) evaluated at $\overline{x}^{\;-1}$,  we obtain the desired result.
\end{pf}
\begin{rem} The previous theorem for the operator $\mathscr{D}_{q}$ yields the well-known hypergeometric equation for the $q$--classical orthogonal polynomials related to weights satisfying the Pearson equation  \begin{equation*}
\mathscr{D}_{q}(A \varrho)(x)=\Psi(x)\varrho(x)\,,\label{eq:pearson-q}
\end{equation*}
given by
  \begin{equation*}
 A({x})\left(\mathscr{D}_{q}\mathscr{D}_{\frac{1}{q}} P_{n+1}\right)(x)
+\Psi({x})(\mathscr{D}_{q}P_{n+1})({x})+\lambda_{n+1}  P_{n+1}({x})=0\,.\label{eq:hyper-Dq}
\end{equation*}
\end{rem}

\subsection{Fourth--order difference equation for Laguerre--Hahn Ortho\-go\-nal Polynomials}\label{sec:5.2}

We now consider the non--semiclassical case in  equation (\ref{eq:defiRic-S}) with  $B \not \equiv 0$, i.e.,  the corresponding sequence of polynomials $\{P_n\}_{n\geq 0}$ is not semiclassical. Our starting point to deduce the fourth--order difference equation satisfied by these polynomials will be equation (\ref{eq:key3}). The following statement will  play a relevant role.

\begin{lema}
Let $\{P_n\}_{n\geq 0}$ be a SMOP related to a Stieltjes function  $S$  satisfying (\ref{eq:defiRic-S}),
\begin{equation*}
A(\overline{x}) (\mathscr{D}_{q,\omega} S)(x) = B(x)  S(x) S(\overline{x}) + C_1(x) S(x) + C_2(x) S(\overline{x}) + D(x) \,,\
\end{equation*}
where  $A, B, C_1, C_2, D$ are polynomials such that $A\not\equiv 0$ and $B\not\equiv 0$, and let $\{P_n^{(1)}\}_{n\geq 0}$ be the monic sequence of associated polynomials. Then, the following equation takes place,
\begin{equation}
\check{A}_n(\overline{x})(\mathscr{D}_{q,\omega}P^{(1)}_{n-1})(\overline{x})=
e_{n,1}(\overline{x})P^{(1)}_{n-1}(\overline{x})+e_{n,2}(\overline{x})P^{(1)}_{n}(\overline{x})
+g_n(\overline{x};P_{n+1}), \label{eq:auxD1}
\end{equation}
where
 \begin{eqnarray}
&& \check{A}_n({x})=\Theta_n(x)\tilde{A}_2({x})\left[\tilde{A}_1({x})\tilde{A}_2({x})+ \varsigma^2 ({x})B(x)D(x)\right]\,,\label{Acn}\\
&&e_{n,1}({x})=\Theta_n(x)\tilde{A}_2(x)\mu_n(x)+\eta_n(x)\varsigma(x)B(x)\Theta_n(x)\,,\label{e1n}\\
&&e_{n,2}(x)=\varsigma(x)\eta_n(x)B(x)(L_n+C_2)(x)+\tilde{A}_2(x)\left[{\eta_n(x)B(x)}-\Theta_n(x)\tilde{A}_1(x)\frac{\Theta_{n-1}(x)}{\gamma_n} \right],\label{e2n}\\
&&g_n(x;P_{n+1})=\tilde{A}_2(x)g_{n,1}(x;P_{n+1})+\eta_n(x)\varsigma(x)B(x)D(x)P_{n+1}({\overline{x}})\label{gn},
 \end{eqnarray}
with
 \begin{align}
&\mu_n(x)=\tilde{A}_1(x)(\tilde{L}_{n-1}+C_2)(x)- \varsigma (x) B(x)D(x)\,,\label{mun}\\
&\eta_n(x)= \tilde{A}_1(x)D(x)+ \varsigma(x)D(x)(\tilde{L}_{n-1}-C_1)(x)\,,\label{etan}\\
&g_{n,1}(x;P_{n+1})=\eta_n(x)\left[\tilde{A}_1(x)(\mathscr{D}_{q,\omega}P_{n+1})(x)-(L_n-C_1)(x)P_{n+1}(x)\right]\notag\\
&\hspace{6cm}-\Theta_n(x)\varsigma(x)D(x)\frac{\Theta_{n-1}(x)}{\gamma_n}P_{n+1}(x)\notag\,.
 \end{align}
 \end{lema}
\begin{pf}
We begin by evaluating equation~(\ref{eq:(22)1}) at $\overline{x}$. Then, using the invariancy property~(\ref{eq:inv}) for $P_{n}(\overline{\overline{x}})$, we obtain
\begin{equation*}
\tilde{A}_2(\overline{x})(\mathscr{D}_{q,\omega}P^{(1)}_{n-1})(\overline{x})=(\tilde{L}_{n-1}+C_2)(\overline{x})P^{(1)}_{n-1}(\overline{x}) -\frac{\Theta_{n-1}(\overline{x})}{\gamma_n}P^{(1)}_{n}(\overline{x})+D(\overline{x})\left(P_{n}(\overline{x})
+\varsigma(\overline{x})(\mathscr{D}_{q,\omega}P_n)(\overline{x})\right).
\end{equation*}
Next, multiplying by $\tilde{A}_1(\overline{x})$, applying formula~(\ref{eq:(21)1}) evaluated at $\overline{x}$, and using  again the invariancy property~(\ref{eq:inv}) for $P_{n-1}^{(1)}(\overline{\overline{x}})$, we have that
\begin{align*}
&\left(\tilde{A}_{1}(\overline{x})\tilde{A}_{2}(\overline{x})+\varsigma^2(\overline{x})B(\overline{x})D(\overline{x})\right)(\mathscr{D}_{q,\omega}P^{(1)}_{n-1})(\overline{x})\\
&=\mu_n(\overline{x})P_{n-1}^{(1)}(\overline{x})-\tilde{A}_{1}(\overline{x})\frac{\Theta_{n-1}(\overline{x})}{\gamma_n}P_{n}^{(1)}(\overline{x})+\eta_n(\overline{x})P_n(\overline{x})-D(\overline{x})\varsigma(\overline{x})\frac{\Theta_{n-1}(\overline{x})}{\gamma_n}P_{n+1}(\overline{x}),
\end{align*}
where $\mu_n(x)$ and $\eta_n(x)$ are defined by~(\ref{mun}) and~(\ref{etan}), respectively. Now, multiplying by $\Theta_n(\overline{x})$ and substituting $\Theta_n(\overline{x})P_n(\overline{x})$ from equation~(\ref{eq:(11)1}), we eliminate $P_n(\overline{x}).$  Then, using (\ref{eq:inv}) for $P_{n}^{(1)}(\overline{\overline{x}})$,  we obtain the following expression
\begin{align*}
&\left(\tilde{A}_{1}(\overline{x})\tilde{A}_{2}(\overline{x})+\varsigma^2(\overline{x})B(\overline{x})D(\overline{x})\right)\Theta_n(\overline{x})(\mathscr{D}_{q,\omega}P^{(1)}_{n-1})(\overline{x})\\
&=\mu_n(\overline{x})\Theta_n(\overline{x})P_{n-1}^{(1)}(\overline{x})+\left(\eta_n(\overline{x})B(\overline{x})-\tilde{A}_1(\overline{x})\frac{\Theta_{n-1}(\overline{x})}{\gamma_n}\Theta_n(\overline{x})\right)P_n^{(1)}(\overline{x})\\
&\,\,+\eta_n(\overline{x})B(\overline{x})\varsigma(\overline{x})(\mathscr{D}_{q,\omega}P^{(1)}_{n})(\overline{x})-\left(\eta_n(\overline{x})(L_n-C_1)(\overline{x})+\varsigma(\overline{x})D(\overline{x})\frac{\Theta_{n-1}(\overline{x})}{\gamma_n}\Theta_n(\overline{x})\right)P_{n+1}(\overline{x})\\
&\,\,+\eta_n(\overline{x})\tilde{A}_{1}(\overline{x})(\mathscr{D}_{q,\omega}P_{n+1})(\overline{x}).
\end{align*}

Finally, multiplying by $\tilde{A}_{2}(\overline{x})$ and applying again~(\ref{eq:(12)1}), the result follows.
\end{pf}

We now have all the necessary results to derive the fourth--order difference equation satisfied by the sequence $\{P_n\}_{n \geq 0}$.
\begin{teo} \label{teo4}
Let $\{P_n\}_{n\geq 0}$ be a SMOP related to a Stieltjes function  $S$  satisfying (\ref{eq:defiRic-S}),
\begin{equation*}
A(\overline{x}) (\mathscr{D}_{q,\omega} S)(x) = B(x)  S(x) S(\overline{x}) + C_1(x) S(x) + C_2(x) S(\overline{x}) + D(x) \,,\
\end{equation*}
where  $A, B, C_1, C_2, D$ are polynomials such that $A\not\equiv 0$ and $B\not\equiv 0$. Then, $\{P_n\}_{n\geq 0}$  satisfies the fourth--order  difference equation
\begin{multline}
\hat{A}_n(x)(\mathscr{D}^{4}_{q,\omega}P_{n+1})(x)+\hat{B}_n(x)(\mathscr{D}^{3}_{q,\omega}P_{n+1})(x)\\
+\hat{C}_n(x)(\mathscr{D}^{2}_{q,\omega}P_{n+1})(x)+\hat{D}_n(x)(\mathscr{D}_{q,\omega}P_{n+1})
(x)+\hat{E}_n(x)P_{n+1}(x)=0,\,~\label{foequation}
\end{multline}
where $\hat{A}_n,\hat{B}_n,\hat{C}_n,\hat{D}_n,\hat{E}_n$  are given explicitly by formulas~(\ref{Ahn})--(\ref{Ehn}) in Appendix~\ref{Appendix}.
\end{teo}
\begin{pf}
Let us write equation (\ref{eq:key3}) as
\begin{equation}
F_0(x;P_{n+1})=U_n(x)P^{(1)}_{n-1}(\overline{x})+V_n(x) P^{(1)}_{n}(\overline{x}) + W_n(x)\left(\mathscr{D}_{q,\omega}P^{(1)}_{n}\right)(\overline{x})\,, \label{eq:4.0}
\end{equation}
where
\begin{align}
F_0(x;P_{n+1})= \Theta_n(x)\tilde{A}_1(x)\tilde{A}_1({\overline{x}})(\mathscr{D}^2_{q,\omega}P_{n+1})(x) +J_n(x)(\mathscr{D}_{q,\omega}P_{n+1})(x)+
K_n(x) P_{n+1}(x)\,,\label{F0}
\end{align}
and
\begin{align*}
U_n(x)&=-B(x)\Theta_{n}(x)\Theta_{n}(\overline{x})\,,\\
V_n(x)&=B(x)[\tilde{A}_1(x)(\mathscr{D}_{q,\omega}\Theta_n)(x)+\Theta_{n}(\overline{x}) (\tilde{L}_{n-1}-C_1)(x)]- \tilde{A}_1(x)\Theta_{n}(x)(\mathscr{D}_{q,\omega}B)(x),\\
W_n(x)&=-q \tilde{A}_1(x)B(\overline{x})\Theta_{n}(x)\,.
\end{align*}

By multiplying equation (\ref{eq:4.0}) by $\tilde{A}_2(\overline{x})$, and using (\ref{eq:(12)1}) evaluated at $\overline{x}$ we eliminate the term $\mathscr{D}_{q,\omega}P^{(1)}_{n}$, and obtain
\begin{multline*}
\tilde{A}_2(\overline{x})F_0(x;P_{n+1})-W_n(x)D(\overline{x})P_{n+1}(\overline{\overline{x}})
=\left\{\tilde{A}_2(\overline{x})U_n(x)+W_n(x)\Theta_n(\overline{x})\right\}P^{(1)}_{n-1}(\overline{x})\\ +\left\{\tilde{A}_2(\overline{x})V_n(x)+W_n(x)(L_n+C_2)(\overline{x})\right\}P^{(1)}_{n}(\overline{x})\,. \label{eq:5}
\end{multline*}

Hence, we have
 \begin{equation}
F_1(x;P_{n+1})=G_{1,1}(x;n)P_{n-1}^{(1)}(\overline{x})+G_{1,2}(x;n)P_{n}^{(1)}(\overline{x})\,, \label{eq:(61)}
\end{equation}
with
\begin{eqnarray}
&&F_1(x;P_{n+1})=\tilde{A}_2(\overline{x})F_0(x;P_{n+1})-W_n(x)D(\overline{x})P_{n+1}(\overline{\overline{x}})\,,\label{F1}\\
&&G_{1,1}(x;n)=\tilde{A}_2(\overline{x})U_n(x)+W_n(x)\Theta_n(\overline{x})\,,\notag\\
&&G_{1,2}(x;n)=\tilde{A}_2(\overline{x})V_n(x)+W_n(x)(L_n+C_2)(\overline{x})\,.\notag
\end{eqnarray}

We now proceed to obtain two more equations in the unknowns  $P_{n-1}^{(1)}(\overline{x})$ and $P_{n}^{(1)}(\overline{x})$.

Taking derivatives in (\ref{eq:(61)}) and multiplying the resulting equation by $\check{A}_n(\overline{x})\tilde{A}_2(\overline{x})$ we obtain, after using  equation (\ref{eq:auxD1}), as well as (\ref{eq:(12)1}) evaluated at $\overline{x}$,
 \begin{equation}
F_2(x;P_{n+1})=G_{2,1}(x;n)P_{n-1}^{(1)}(\overline{x})+G_{2,2}(x;n)P_{n}^{(1)}(\overline{x})\,, \label{eq:(71)}
\end{equation}
with
\begin{align}
F_2(x;P_{n+1})&=\tilde{A}_2(\overline{x})\check{A}_n(\overline{x})\mathscr{D}_{q,\omega}F_1(x;P_{n+1})-q G_{1,1}(\overline{x};n)\tilde{A}_2(\overline{x})g_n(\overline{x};P_{n+1})\notag\\
& -q  G_{1,2}(\overline{x};n)\check{A}_n(\overline{x})D(\overline{x})P_{n+1}(\overline{\overline{x}})\,,\label{F2}\\
G_{2,1}(x;n)&=\tilde{A}_2(\overline{x})\check{A}_n(\overline{x})\mathscr{D}_{q,\omega}G_{1,1}(x;n)
+q G_{1,1}(\overline{x};n)\tilde{A}_2(\overline{x})e_{n,1}(\overline{x})+q G_{1,2}(\overline{x};n)\check{A}_n(\overline{x})\Theta_n(\overline{x})\,,\notag\\
G_{2,2}(x;n)&=\tilde{A}_2(\overline{x})\check{A}_n(\overline{x})\mathscr{D}_{q,\omega}G_{1,2}(x;n)+
q G_{1,1}(\overline{x};n)\tilde{A}_2(\overline{x})e_{n,2}(\overline{x})+q G_{1,2}(\overline{x};n)\check{A}_n(\overline{x})(L_n+C_2)(\overline{x})\,.\notag
\end{align}

Analogously, taking derivatives in (\ref{eq:(71)}) and multiplying the resulting equation by $\check{A}_n(\overline{x})\tilde{A}_2(\overline{x})$ we obtain
\begin{equation}
F_3(x;P_{n+1})=G_{3,1}(x;n)P_{n-1}^{(1)}(\overline{x})+G_{3,2}(x;n)P_{n}^{(1)}(\overline{x})\,, \label{eq:(81)}
\end{equation}
with
\begin{align}
F_3(x;P_{n+1})=&\tilde{A}_2(\overline{x})\check{A}_n(\overline{x})\mathscr{D}_{q,\omega}F_2(x;P_{n+1})-q G_{2,1}(\overline{x};n)\tilde{A}_2(\overline{x})g_n(\overline{x};P_{n+1})\notag\\
&-qG_{2,2}(\overline{x};n)\check{A}_n(\overline{x})D(\overline{x})P_{n+1}(\overline{\overline{x}})\,,\label{F3}\\
G_{3,1}(x;n)=&\tilde{A}_2(\overline{x})\check{A}_n(\overline{x})\mathscr{D}_{q,\omega}G_{2,1}(x;n)
+q G_{2,1}(\overline{x};n)\tilde{A}_2(\overline{x})e_{n,1}(\overline{x})+q G_{2,2}(\overline{x};n)\check{A}_n(\overline{x})\Theta_n(\overline{x})\,,\notag\\
G_{3,2}(x;n)=&\tilde{A}_2(\overline{x})\check{A}_n(\overline{x})\mathscr{D}_{q,\omega}G_{2,2}(x;n)+
q G_{2,1}(\overline{x};n)\tilde{A}_2(\overline{x})e_{n,2}(\overline{x})+q G_{2,2}(\overline{x};n)\check{A}_n(\overline{x})(L_n+C_2)(\overline{x})\,.\notag
\end{align}

Henceforth, we have obtained a system of three equations in two unknowns $P_{n-1}^{(1)}(\overline{x})$ and $P_{n}^{(1)}(\overline{x})$, composed by equations (\ref{eq:(61)}), (\ref{eq:(71)}), (\ref{eq:(81)}), that is,
\begin{equation*}
\begin{cases}
G_{1,1}(x;n)P_{n-1}^{(1)}(\overline{x})+G_{1,2}(x;n)P_{n}^{(1)}(\overline{x})=F_1(x;P_{n+1}),\\
G_{2,1}(x;n)P_{n-1}^{(1)}(\overline{x})+G_{2,2}(x;n)P_{n}^{(1)}(\overline{x})=F_2(x;P_{n+1}),\\
G_{3,1}(x;n)P_{n-1}^{(1)}(\overline{x})+G_{3,2}(x;n)P_{n}^{(1)}(\overline{x})=F_3(x;P_{n+1})\,.
\end{cases}
\end{equation*}

Thus, the fourth--order difference equation is given by the determinant
\begin{equation*}
\det \begin{bmatrix} G_{1,1} & G_{1,2} &F_1(x;P_{n+1})\\
G_{2,1} & G_{2,2} & F_2(x;P_{n+1})\\
G_{3,1} & G_{3,2} & F_3(x;P_{n+1})
\end{bmatrix}=0\,.
\end{equation*}
By expanding the determinant along the third column we get
\begin{equation*}
I_{3}(x;n)F_3(x;P_{n+1})-I_{2}(x;n) F_2(x;P_{n+1})+ I_{1}(x;n) F_1(x;P_{n+1})=0\,,\label{eq:3col-expand}
\end{equation*}
with
$$I_{3}=\det \begin{bmatrix}
G_{1,1} & G_{1,2} \\
G_{2,1} & G_{2,2}
\end{bmatrix}\,, \; I_{2}=\det \begin{bmatrix}
G_{1,1} & G_{1,2} \\
G_{3,1} & G_{3,2}
\end{bmatrix}\,, \; I_{1}=\det \begin{bmatrix}
G_{2,1} & G_{2,2} \\
G_{3,1} & G_{3,2}
\end{bmatrix}\,.
$$
 Finally, to obtain the explicit expressions for the coefficients $\hat{A}_n,\hat{B}_n,\hat{C}_n,\hat{D}_n,\hat{E}_n$ of the fourth--order difference equation, we expand the above equation taking into account that from definitions (\ref{F1}),~(\ref{F2}),~(\ref{F3}), and using~(\ref{eq:inv}),~(\ref{gn}),~(\ref{F0}), then, $F_1$, $F_2$, $F_3$ can be expressed as
\begin{align*}
F_1(x;P_{n+1})=&f_{1,2}(x;n)(\mathscr{D}^2_{q,\omega}P_{n+1})(x)+f_{1,1}(x;n)(\mathscr{D}_{q,\omega}P_{n+1})(x)+f_{1,0}(x;n)P_{n+1}(x),\\
F_2(x;P_{n+1})=&f_{2,3}(x;n)(\mathscr{D}^3_{q,\omega}P_{n+1})(x)+f_{2,2}(x;n)(\mathscr{D}^2_{q,\omega}P_{n+1})(x)+f_{2,1}(x;n)(\mathscr{D}_{q,\omega}P_{n+1})(x)\\
&+f_{2,0}(x;n)P_{n+1}(x),\\
F_3(x;P_{n+1})=&f_{3,4}(x;n)(\mathscr{D}^4_{q,\omega}P_{n+1})(x)+f_{3,3}(x;n)(\mathscr{D}^3_{q,\omega}P_{n+1})(x)+f_{3,2}(x;n)(\mathscr{D}^2_{q,\omega}P_{n+1})(x)\\
&+f_{3,1}(x;n)(\mathscr{D}_{q,\omega}P_{n+1})(x)+f_{3,0}(x;n)P_{n+1}(x),
\end{align*}
where $f_{i,j}$, $i=1,2,3$, $j=0,\dots,i+1$, are given by formulas~(\ref{f10})--(\ref{f34}) in the Appendix.

\end{pf}

\section{Examples} \label{s-examples}

The explicit expressions for the coefficients of the fourth--order difference equation~(\ref{foequation}), can be computed through expressions~(\ref{Ahn})--(\ref{Ehn}) from the Appendix, using any programming language.

In this section, we deduce the coefficients of this equation for three families of orthogonal polynomials corresponding to different cases of the Hahn difference operator. In what follows, the \emph{associated polynomials} will play the role of $P_n$ throughout this paper, rather than that of  $P_n^{(1)}$. The purpose of this is to compute the fourth--order difference equation for these polynomials using Theorem~\ref{teo4}.

We have chosen these families because the computation of the polynomials $\Theta_n$ and $L_n$ is laborious, whereas the aim of this section is to illustrate that the analytical results obtained in previous sections hold for well--known families of orthogonal polynomials. All computations were performed using \ma\, 14.0.

\subsection{Associated Laguerre orthogonal polynomials}
\label{assla}

We consider the monic associated Laguerre orthogonal polynomials which are related to the derivative operator.

First, we recall that for the Laguerre weight $\varrho(x)=\frac{x^\alpha e^{-x}}{\Gamma(\alpha+1)},\, \alpha>-1\,$, the following Pearson equation holds, see~\cite[p. 23]{nik-MIR}
$$(x\varrho(x))'=(-x+\alpha+1)\varrho(x).$$
Notice that $\varrho_0=1$. From Lemma~\ref{lemma:ric-sc}, the Stieltjes function $S$ for the monic Laguerre polynomials satisfies equation~(\ref{eq:ric-S-sc}) with coefficients
\begin{equation*}
A(x)=x,\quad C(x)=-x+\alpha,\quad D(x)=1.
\end{equation*}
Thus, from relations~(\ref{A1rel})--(\ref{D1rel}), taking into account that $B=0$, $C_1=C$, and that the recurrence coefficients in~(\ref{eq:ttrr-Pn}) for the monic Laguerre orthogonal polynomials, see~\cite[p. 38]{nik-MIR}, are given by
\begin{equation*}
\beta_n=2n+\alpha+1,\quad \gamma_n=n(n+\alpha),
\end{equation*}
we deduce that the Stieljes function $S_1$ for the monic associated Laguerre orthogonal polynomials satisfies equation (\ref{eq:defiRic-S}) with coefficients
\begin{equation*}
A^{(1)}(x)=x,\quad B^{(1)}(x)=\alpha+1,\quad C_1^{(1)}(x)=-x+\alpha+2,\quad C_2^{(1)}(x)=0,\quad D^{(1)}(x)=1.
\end{equation*}

As explained in the introduction of this section, to compute the fourth--order differential equation satisfied by the monic associated Laguerre orthogonal polynomials, we will apply Theorem~\ref{teo4} with $P_n$ replaced by $P_n^{(1)}$. Thus, we also need to deduce the polynomials $\Theta_n$ and $L_n$ in formulas (\ref{eq:est-Pn})--(\ref{eq:est-Pn(1)}) with $P_n$ replaced by $P_n^{(1)}$, and consequently $P_n^{(1)}$ replaced by the associated polynomials of $P_n^{(1)}$. To avoid confusion we denote these polynomials by $\breve{\Theta}_n$ and $\breve{L}_n$. Comparing terms in equation (\ref{eq:est-Pn}) with these considerations, we obtain that
\begin{equation*}
\breve{L}_n(x)=-x+n+\alpha+3,\quad \breve{\Theta}_n(x)=(n+2)(n+\alpha+2).
\end{equation*}

Therefore, taking into account formulas~(\ref{foequation}),~(\ref{Ahn})--(\ref{Ehn}), with the role of $P_n$, $A$, $B$, $C_1$, $C_2$, $D$, $\Theta_n$, $L_n$ played by $P_n^{(1)},\,A^{(1)},\, B^{(1)},\,C_1^{(1)},\,C_2^{(1)},\,D^{(1)},\,\breve{\Theta}_n, \breve{L}_n$, respectively, we obtain that  the monic associated Laguerre orthogonal polynomial $P_{n+1}^{(1)}$ satisfies the fourth--order differential equation
\begin{align*}
&-x^2y^{(4)}(x)-5xy^{(3)}(x)+(x^2-2x(n+\alpha+3)+\alpha^2-4)y''(x)\\
&\quad+3(x-n-\alpha-3)y'(x)-(n+3)(n+1)y(x)=0.
\end{align*}

\subsection{Associated Charlier orthogonal polynomials}

We consider the monic associated Charlier orthogonal polynomials, which are related to the forward difference operator.

We begin recalling that for the Charlier weight $\varrho(x)=\frac{e^{-a}a^x}{\Gamma(x+1)},\,a>0$, the following Pearson equation holds, see~\cite[p. 44]{niki-sus-uv}
$$\Delta(x\varrho(x))=(a-x)\varrho(x).$$
Notice that $\varrho_0=1$. Consequently, from Lemma~\ref{lemma:ric-sc}, the Stieltjes function $S$ for the monic Charlier polynomials satisfies equation~(\ref{eq:ric-S-sc}) with coefficients
$$A(x)=x,\quad C(x)=-x+a-1,\quad D(x)=1.$$
Furthermore, the recurrence coefficients in~(\ref{eq:ttrr-Pn}) for the monic Charlier polynomials, see~\cite[p. 44]{niki-sus-uv} are given by
\begin{equation*}
\beta_n=n+a,\quad \gamma_n=n a.
\end{equation*}

Proceeding as in the previous example, with the same considerations and notation, see Section~\ref{assla}, we deduce that the Stieltjes function $S_1$ for the monic associated Charlier orthogonal polynomials  satisfies equation (\ref{eq:defiRic-S}) with coefficients
\begin{equation*}
A^{(1)}(x)=x,\quad B^{(1)}(x)=a,\quad C_1^{(1)}(x)=-x+a-1,\quad C_2^{(1)}(x)=1,\quad D^{(1)}(x)=1,
\end{equation*}
and that the polynomials $\breve{L}_n$, $\breve{\Theta}_n$ introduced in Section~\ref{assla} are given by
\begin{equation*}
\breve{L}_n(x)=-x+a-1,\quad \breve{\Theta}_n(x)=(n+2)a.
\end{equation*}
Therefore, taking into account formulas~(\ref{foequation}),~(\ref{Ahn})--(\ref{Ehn}) with the same considerations as in Section~\ref{assla}, we obtain that the monic associated Charlier orthogonal polynomials $P_{n+1}^{(1)}$ satisfies the fourth--order difference equation
\begin{equation*}
\alpha_{4,n}(x)\Delta^4y(x)+\alpha_{3,n}(x)\Delta^3y(x)+\alpha_{2,n}(x)\Delta^2y(x)+\alpha_{1,n}(x)\Delta y(x)+\alpha_{0,n}y(x)=0,
\end{equation*}
with coefficients
\begin{align*}
{\alpha}_{4,n}(x)&=-a(x+4)(2x-n+2a+1),\\
{\alpha}_{3,n}(x)&= 2 x^3-(3 n+2 a-9)x^2+ \left(n^2-11 n-2 a n-2 a^2-20 a+10\right)x\\
&+3n^2+a n^2-3 a^2 n+5 a n-6 n+2 a^3-23 a^2-12 a+3,\\
{\alpha}_{2,n}(x)&=2 x^3-(3n+2a-15)x^2-\left(n^2+25n+2 a n+2 a^2+16 a-16\right)x\\
&+n^2 (n+7)-15n-3 a^2 n-a n^2-6 a n+2 a^3-23 a^2-17 a+3,\\
{\alpha}_{1,n}(x)&=6x^2-2n(2n+11)x+2n(n^2+3n-7)-4(n+1)(n+3)a-6a^2-6,\\
{\alpha}_{0,n}(x)&=-2x(n+1)(n+3)+(n+3)(n+1)(n-2a-2).\\
\end{align*}

\subsection{Associated little $q$--Laguerre orthogonal polynomials}

We consider the monic associated little $q$--Laguerre orthogonal polynomials, which are related to the Jackson $q-$difference operator $\mathscr{D}_q$.

The monic little $q$--Laguerre polynomials are the polynomials orthogonal with respect to the following inner product, see~\cite{AGMMB-00} and \cite[Sec. 14.20]{koek}
\begin{equation}
\label{INPLQL}
( p,r)=\sum_{k=0}^\infty\frac{(aq)^k(aq;q)_\infty}{(q;q)_k}p(q^k)r(q^k),\quad 0<aq<1,\quad \text{ for every } p,r\in\mathbb{P},
\end{equation}
where the $q-$shifted factorials are defined by~\cite[p. 11]{koek}
\begin{equation*}
(b;q)_0=1,\quad (b;q)_k=\prod_{j=1}^k(1-bq^{j-1}),\quad k\geq1,\quad \text{and}\quad (b;q)_\infty=\prod_{j=1}^\infty(1-bq^{j-1}).
\end{equation*}
Using the $q-$integral, see~\cite[f. (3)]{anaby} and $(q;q)_k=(q;q)_\infty/(q^{k+1};q)_\infty$, the previous inner product~(\ref{INPLQL}) can be written for $a=q^\alpha$, $\alpha>-1$, as
\begin{equation*}
(p,r)=\int_{0}^1p(x)r(x)\varrho(x)d_q(x),
\end{equation*}
with $\displaystyle\varrho(x)=\frac{(q^{\alpha+1};q)_\infty x^\alpha(qx;q)_\infty}{(1-q)(q;q)_\infty}.$ Notice that $\varrho_0=1$, see~\cite[f. (14.20.2)]{koek}. For this weight, the following Pearson equation holds, see~\cite{medem} $$\mathscr{D}_q\left(x(1-x)\varrho(x)\right)=\left(\frac{-x-q^{\alpha+1}+1}{1-q}\right)\varrho(x).$$
Thus, according to Lemma~\ref{lemma:ric-sc}, the Stieltjes function for the little $q-$Laguerre orthogonal polynomials satisfies equation (\ref{eq:defiRic-S}) with coefficients
\begin{equation*}
A(x)=x(1-x),\quad C(x)=\frac{q(-qx-q^\alpha+1)}{1-q},\quad D(x)=\frac{q}{1-q}.
\end{equation*}

On the other hand, the recurrence coefficients in~(\ref{eq:ttrr-Pn}) for the little $q-$Laguerre orthogonal polynomials, see \cite[f. (14.20.3)]{koek} are given by
\begin{align*}
\beta_n=q^{n}(q^\alpha-(q+1)q^{n+\alpha}+1),\quad\gamma_n=q^{2n+\alpha-1}(q^{n+1}-1)(q^{n+\alpha}-1).
\end{align*}

Proceeding as in the example in Section~\ref{assla}, with the same considerations and notation, we deduce that the Stieltjes function $S_1$ for the monic associated little $q$--Laguerre orthogonal polynomials satisfies an equation of type~(\ref{eq:defiRic-S}) with coefficients
\begin{align*}
A^{(1)}(x)&=x(1-x),\quad B^{(1)}(x)=q^{\alpha+2}(-q^{\alpha+1}+1),\quad C_1^{(1)}(x)=-\frac{q}{1-q}(qx+q^{\alpha+1}-1),\\
C_2^{(1)}(x)&=q(-x+q^\alpha),\quad D^{(1)}(x)=\frac{1}{1-q},
\end{align*}
and that the polynomials $\breve{L}_n$ and $\breve{\Theta}_n$ introduced in Section~\ref{assla} are given by
\begin{equation*}
\breve{L}_n(x)=-\frac{q}{1-q}(qx+q^{n+\alpha+2}-1),\quad \breve{\Theta}_n(x)=\frac{q^{n+\alpha+2}(q^{n+2}-1)(q^{n+\alpha+2}-1)}{1-q}.
\end{equation*}

Therefore, from formulas~(\ref{foequation}), (\ref{Ahn})--(\ref{Ehn}), with the same considerations as in Section~\ref{assla}, we obtain that the monic associated little $q-$Laguerre orthogonal polynomial $P_{n+1}^{(1)}$ satisfies the fourth--order difference equation
\begin{equation*}
\alpha_{4,n}(x)\mathscr{D}_q^4y(x)+\alpha_{3,n}(x)\mathscr{D}_q^3y(x)+\alpha_{2,n}(x)\mathscr{D}_q^2y(x)+\alpha_{1,n}(x)\mathscr{D}_q y(x)+\alpha_{0,n}(x)y(x)=0,
\end{equation*}
with coefficients
\begin{align*}
\alpha_{4,n}(x)&=(q-1)^4 x^4 \left(q^{n+2}+1\right) q^{\alpha +2 n+10}-(q-1)^4 x^3 \left((q+1) q^{\alpha
   +n+4}+q^{n+2}+q^{n+4}\right.\\
   &\left.+q^{n+5}+1\right) q^{\alpha +2 n+6}+(q-1)^4 (q+1) x^2 \left(q^{\alpha }+1\right)
   q^{\alpha +3 n+6},\\
\alpha_{3,n}(x)&=(q-1)^3 x^4 \left(q^{n+2}+1\right) q^{n+7}+(q-1)^3 (q+1) x q^{3
   n+3} \left(q^{\alpha }+1\right) \left(-q^{2 \alpha }+q^{\alpha
   +2}+q^{\alpha +3}-1\right)\\
   &+(q-1)^3 x^2 q^{2 n+3} \left(q^{\alpha
   +n+2}-2 q^{\alpha +n+5}-(q (q+2)+2) q^{\alpha +n+6}+q^{2 \alpha
   +n+2}+q^{n+2}+q^{n+3}+q^{n+4}\right.\\
   &\left.-(q+1) \left(q^2+q+1\right) q^{2
   \alpha +n+5}+q^{2 \alpha }+q^{\alpha +1}+q^{\alpha +2}-q^{\alpha
   +3}+(q+1) q^{2 \alpha +1}+q^{\alpha }+q^2+q+1\right)\\
   &+(q-1)^3 x^3
   q^{n+4} \left(-2 q^{n+2}-q^{n+3}-q^{n+4}-q^{2
   n+4}+\left(q^4+q^3-2\right) q^{\alpha
   +n+2}\right.\\
   &\left.+\left(q^4+q^3+q^2+q-1\right) q^{\alpha +2 n+4}-q^{\alpha
   }-1\right),\\
\alpha_{2,n}(x)&=(q-1)^2 q^2 x^3 \left(-q^{n+2}+q^{n+3}+q^{n+4}+q^{n+5}-q^{2
   n+4}+q^{2 n+5}+q^{2 n+6}+q^{2 n+7}-q^{3 n+6}-1\right)\\
   &+(q-1)^2
   (q+1) q^{3 n+1} \left(q^{\alpha }+1\right) \left(q^{\alpha
   +4}-q^{2 \alpha +2}+q^{\alpha }-q^2\right)+(q-1)^2 x^2 q^{n+1}
   \left(q^{\alpha +n+2}+q^{n+2}\right.\\
   &-2 q^{n+4}-4 q^{n+5}-2
   q^{n+6}-q^{n+7}+q^{2 n+4}+q^{2 n+5}-q^{2 n+7}+\left(q
   \left(q^4-3 q-2\right)-1\right) q^{\alpha
   +n+3}\\
   &\left.+\left(\left(q^2+q+2\right) q^4+1\right) q^{\alpha +2
   n+4}-q^3+\left(-q^3+q+1\right) q^{\alpha }+q+1\right)\\
   &+(q-1)^2 x
   q^{2 n+1} \left(-q^{n+2}+q^{n+4}+2 q^{n+5}+q^{n+6}-\left(q^2
   \left(q^3+q-1\right) (q+1)^2+2\right) q^{\alpha +n+2}\right.\\
   &\left.-\left(q
   (q+1) \left(q^2+1\right) \left(q^3+q^2-1\right)+1\right) q^{2
   \alpha +n+2}-q^{2 \alpha }-q^{\alpha +1}+q^{\alpha +2}+3
   q^{\alpha +3}+q^{\alpha +4}\right.\\
   &\left.+(q+1)^2 q^{2 \alpha +2}-2 q^{\alpha
   }+q^4+2 q^3+q^2-1\right),\\
   \alpha_{1,n}(x)&=(q-1) q x^2 \left(q^{n+1}-q^{n+2}+q^{n+4}+q^{n+5}+q^{2 n+3}-q^{2
   n+4}+q^{2 n+6}+q^{2 n+7}-q^{3 n+6}\right.\\
   &\left.-q^{3 n+7}-q-1\right)-(q-1)^2
   \left(q^2+q+1\right) q^{2 n+1} \left(q^{\alpha }+1\right)
   \left(q^{\alpha +n+1}+q^{\alpha +n+3}+q^{\alpha
   +n+4}-q^{n+2}\right.\\
   &\left.-q^{\alpha }-1\right)+(q-1) x q^n \left(-q^{\alpha
   +n+1}-2 q^{\alpha +n+3}-q^{\alpha +n+4}-3 q^{\alpha
   +n+5}-q^{\alpha +n+6}+q^{\alpha +2 n+4}\right.\\
   &\left.+q^{\alpha +2
   n+6}+q^{\alpha +2 n+8}+q^{\alpha +2 n+9}-q^{n+1}-q^{n+3}-3
   q^{n+5}-2 q^{n+6}-q^{n+7}+q^{2 n+4}+q^{2 n+5}\right.\\
   &\left.+2 q^{2
   n+6}+q^{\alpha +1}+2 q^{\alpha +2}+q^{\alpha }+2 q^2+q+1\right),\\
   \alpha_{0,n}(x)&=\left(q^2+1\right) q^n \left(q^{n+1}-1\right)
   \left(q^{n+3}-1\right) \left(q^{\alpha }+1\right)-q x
   \left(q^{n+1}-1\right) \left(q^{n+2}+1\right)
   \left(q^{n+3}-1\right).
\end{align*}

\section*{Acknowledgements}
CRP would like to thank MN for her invitation to do a research stay in University of Beira Interior, Covilhã (Portugal) where this work was initiated.\\
The work of the author CRP is mainly financed by Plan Propio de la Universidad de Almería. Furthermore, the authors JFMM, JJMB and CRP are partially supported by MICIU/AEI/10.13039/501100011033 grant PID2025-170285NB-I00. In addition, the work of JFMM, JJMB and CRP is partially supported by the Research Group FQM-0229 of Universidad de Almer\'{\i}a. The work of MNR is partially supported by CMA-UBI (grants UID-B-MAT/00212/2020 and UID-P-MAT/00212/2020), funded by FCT - Fundação para a Ciência e a Tecnologia.

\newpage
\appendix
\section*{Appendix}
\section{Explicit expressions of the coefficients $\hat{A}_n,\hat{B}_n,\hat{C}_n,\hat{D}_n,\hat{E}_n$}
\label{Appendix}
The coefficients $\hat{A}_n,\hat{B}_n,\hat{C}_n,\hat{D}_n,\hat{E}_n$ of the fourth--order difference equation~(\ref{foequation}) in Theorem~\ref{teo4} can be computed recursively through the following expressions
\begin{align}
\hat{A}_n(x)=&f_{3,4}(x;n)I_3(x;n),\label{Ahn}\\
\hat{B}_n(x)=&-f_{2,3}(x;n)I_2(x;n)+f_{3,3}(x;n)I_3(x;n),\label{Bhn}\\
\hat{C}_n(x)=&f_{1,2}(x;n)I_1(x;n)-f_{2,2}(x;n)I_2(x;n)+f_{3,2}(x;n)I_3(x;n),\label{Chn}\\
\hat{D}_n(x)=&f_{1,1}(x;n)I_1(x;n)-f_{2,1}(x;n)I_2(x;n)+f_{3,1}(x;n)I_3(x;n),\label{Dhn}\\
\hat{E}_n(x)=&f_{1,0}(x;n)I_1(x;n)-f_{2,0}(x;n)I_2(x;n)+f_{3,0}(x;n)I_3(x;n)\label{Ehn},
\end{align}
with $f_{i,j},$ $i=1,2,3$, $j=0,\dots,i+1$, given by
\begin{align}
f_{1,0}(x;n)=&\tilde{A}_2(\overline{x})K_n(x)-W_n(x)D(\overline{x}),\label{f10}\\
f_{1,1}(x;n)=&\tilde{A}_2(\overline{x})J_n(x)-(1+q)\varsigma(x)W_n(x)D(\overline{x}),\label{f11}\\
f_{1,2}(x;n)=&\tilde{A}_2(\overline{x})\Theta_n(x)\tilde{A}_1(x)\tilde{A}_1(\overline{x})-q\varsigma^2(x)W_n(x)D(\overline{x}),\label{f11}\\\notag\\
f_{2,0}(x;n)=&\tilde{A}_2(\overline{x})\check{A}_n(\overline{x})\mathscr{D}_{q,\omega}f_{1,0}(x;n)\notag\\
&+qG_{1,1}(\overline{x};n)(\tilde{A}_2(\overline{x}))^2\left(\eta_n(\overline{x})(L_n-C_1)(\overline{x})+\Theta_n(\overline{x})\varsigma(\overline{x})D(\overline{x})\frac{\Theta_{n-1}(\overline{x})}{\gamma_n}\right)\notag\\
&-qD(\overline{x})\left(G_{1,1}(\overline{x};n)\tilde{A}_2(\overline{x})\eta_n(\overline{x})\varsigma(\overline{x})B(\overline{x})+G_{1,2}(\overline{x};n)\check{A}_n(\overline{x})\right)\label{f20},\\
f_{2,1}(x;n)=&\tilde{A}_2(\overline{x})\check{A}_n(\overline{x})\left(\mathscr{D}_{q,\omega}f_{1,1}(x;n)+f_{1,0}(\overline{x};n)\right)\notag\\
&+q\varsigma(x)(\tilde{A}_2(\overline{x}))^2G_{1,1}(\overline{x};n)\left(\eta_n(\overline{x})(L_n-C_1)(\overline{x})+\Theta_n(\overline{x})\varsigma(\overline{x})D(\overline{x})\frac{\Theta_{n-1}(\overline{x})}{\gamma_n}\right)\notag\\
&-q(1+q)\varsigma(x)D(\overline{x})\left(G_{1,1}(\overline{x};n)\tilde{A}_2(\overline{x})\eta_n(\overline{x})\varsigma(\overline{x})B(\overline{x})+G_{1,2}(\overline{x};n)\check{A}_n(\overline{x})\right)\notag\\
&-qG_{1,1}(\overline{x};n)\tilde{A}_1(\overline{x})(\tilde{A}_2(\overline{x}))^2\eta_n(\overline{x}),\label{f21}\\
f_{2,2}(x;n)=&\tilde{A}_2(\overline{x})\check{A}_n(\overline{x})\left(\mathscr{D}_{q,\omega}f_{1,2}(x;n)+f_{1,1}(\overline{x};n)\right)\notag\\
&-q^2\varsigma^2(x)D(\overline{x})\left(G_{1,1}(\overline{x};n)\tilde{A}_2(\overline{x})\eta_n(\overline{x})\varsigma(\overline{x})B(\overline{x})+G_{1,2}(\overline{x};n)\check{A}_n(\overline{x})\right)\notag\\
&-q\varsigma(x)G_{1,1}(\overline{x};n)\tilde{A}_1(\overline{x})(\tilde{A}_2(\overline{x}))^2\eta_n(\overline{x}),\label{f21}\\
f_{2,3}(x;n)=&\tilde{A}_2(\overline{x})\check{A}_n(\overline{x})f_{1,2}(\overline{x};n),\label{f23}\\\notag\\
f_{3,0}(x;n)=&\tilde{A}_2(\overline{x})\check{A}_n(\overline{x})\mathscr{D}_{q,\omega}f_{2,0}(x;n)\notag\\
&+qG_{2,1}(\overline{x};n)(\tilde{A}_2(\overline{x}))^2\left(\eta_n(\overline{x})(L_n-C_1)(\overline{x})+\Theta_n(\overline{x})\varsigma(\overline{x})D(\overline{x})\frac{\Theta_{n-1}(\overline{x})}{\gamma_n}\right)\notag\\
&-qD(\overline{x})\left(G_{2,1}(\overline{x};n)\tilde{A}_2(\overline{x})\eta_n(\overline{x})\varsigma(\overline{x})B(\overline{x})+G_{2,2}(\overline{x};n)\check{A}_n(\overline{x})\right),\label{f30}\\
f_{3,1}(x;n)=&\tilde{A}_2(\overline{x})\check{A}_n(\overline{x})\left(\mathscr{D}_{q,\omega}f_{2,1}(x;n)+f_{2,0}(\overline{x};n)\right)\notag\\
&+q\varsigma(x)(\tilde{A}_2(\overline{x}))^2G_{2,1}(\overline{x};n)\left(\eta_n(\overline{x})(L_n-C_1)(\overline{x})+\Theta_n(\overline{x})\varsigma(\overline{x})D(\overline{x})\frac{\Theta_{n-1}(\overline{x})}{\gamma_n}\right)\notag\\
&-q(1+q)\varsigma(x)D(\overline{x})\left(G_{2,1}(\overline{x};n)\tilde{A}_2(\overline{x})\eta_n(\overline{x})\varsigma(\overline{x})B(\overline{x})+G_{2,2}(\overline{x};n)\check{A}_n(\overline{x})\right)\notag\\
&-qG_{2,1}(\overline{x};n)\tilde{A}_1(\overline{x})(\tilde{A}_2(\overline{x}))^2\eta_n(\overline{x}),\label{f31}\\
f_{3,2}(x;n)=&\tilde{A}_2(\overline{x})\check{A}_n(\overline{x})\left(\mathscr{D}_{q,\omega}f_{2,2}(x;n)+f_{2,1}(\overline{x};n)\right)\notag\\
&-q^2\varsigma^2(x)D(\overline{x})\left(G_{2,1}(\overline{x};n)\tilde{A}_2(\overline{x})\eta_n(\overline{x})\varsigma(\overline{x})B(\overline{x})+G_{2,2}(\overline{x};n)\check{A}_n(\overline{x})\right)\notag\\
&-q\varsigma(x)G_{2,1}(\overline{x};n)\tilde{A}_1(\overline{x})(\tilde{A}_2(\overline{x}))^2\eta_n(\overline{x}),\label{f32}\\
f_{3,3}(x;n)=&\tilde{A}_2(\overline{x})\check{A}_n(\overline{x})\left(\mathscr{D}_{q,\omega}f_{2,3}(x;n)+f_{2,2}(\overline{x};n)\right),\label{f33}\\
f_{3,4}(x;n)=&\tilde{A}_2(\overline{x})\check{A}_n(\overline{x})f_{2,3}(\overline{x};n),\label{f34}
\end{align}
where
\begin{align*}
U_n(x)&=-B(x)\Theta_{n}(x)\Theta_{n}(\overline{x})\,,\\
V_n(x)&=B(x)[\tilde{A}_1(x)(\mathscr{D}_{q,\omega}\Theta_n)(x)+\Theta_{n}(\overline{x}) (\tilde{L}_{n-1}-C_1)(x)]- \tilde{A}_1(x)\Theta_{n}(x)(\mathscr{D}_{q,\omega}B)(x),\\
W_n(x)&=-q \tilde{A}_1(x)B(\overline{x})\Theta_{n}(x)\,,
\end{align*}
and $I_i$, $i=1,2,3$ are given by
\begin{align*}
I_1(x;n)&=G_{2,1}(x;n)G_{3,2}(x;n)-G_{2,2}(x;n)G_{3,1}(x;n),\\
I_2(x;n)&=G_{1,1}(x;n)G_{3,2}(x;n)-G_{1,2}(x;n)G_{3,1}(x;n),\\
I_3(x;n)&=G_{1,1}(x;n)G_{2,2}(x;n)-G_{1,2}(x;n)G_{2,1}(x;n),\\
\end{align*}
with
\begin{align*}
G_{1,1}(x;n)&=\tilde{A}_2(\overline{x})U_n(x)+W_n(x)\Theta_n(\overline{x})\,,\\
G_{1,2}(x;n)&=\tilde{A}_2(\overline{x})V_n(x)+W_n(x)(L_n+C_2)(\overline{x})\,,\\
G_{2,1}(x;n)&=\tilde{A}_2(\overline{x})\check{A}_n(\overline{x})\mathscr{D}_{q,\omega}G_{1,1}(x;n)
+q G_{1,1}(\overline{x};n)\tilde{A}_2(\overline{x})e_{n,1}(\overline{x})+q G_{1,2}(\overline{x};n)\check{A}_n(\overline{x})\Theta_n(\overline{x})\,,\\
G_{2,2}(x;n)&=\tilde{A}_2(\overline{x})\check{A}_n(\overline{x})\mathscr{D}_{q,\omega}G_{1,2}(x;n)+
q G_{1,1}(\overline{x};n)\tilde{A}_2(\overline{x})e_{n,2}(\overline{x})+q G_{1,2}(\overline{x};n)\check{A}_n(\overline{x})(L_n+C_2)(\overline{x}),\\
G_{3,1}(x;n)&=\tilde{A}_2(\overline{x})\check{A}_n(\overline{x})\mathscr{D}_{q,\omega}G_{2,1}(x;n)
+q G_{2,1}(\overline{x};n)\tilde{A}_2(\overline{x})e_{n,1}(\overline{x})+q G_{2,2}(\overline{x};n)\check{A}_n(\overline{x})\Theta_n(\overline{x})\,,\\
G_{3,2}(x;n)&=\tilde{A}_2(\overline{x})\check{A}_n(\overline{x})\mathscr{D}_{q,\omega}G_{2,2}(x;n)+
q G_{2,1}(\overline{x};n)\tilde{A}_2(\overline{x})e_{n,2}(\overline{x})+q G_{2,2}(\overline{x};n)\check{A}_n(\overline{x})(L_n+C_2)(\overline{x})\,,
\end{align*}
and $C_1,\,C_2,\Theta_n,\,L_n$ are the polynomials given in formulas~(\ref{eq:defiRic-S}), (\ref{eq:est-Pn})--(\ref{eq:est-Pn(1)}), and $\varsigma,\tilde{A}_2,\check{A}_n,e_{n,1},e_{n,2},$ are defined in~(\ref{eq:inv}), (\ref{At12}), (\ref{Acn}), (\ref{e1n}), (\ref{e2n}), respectively.


\begin{thebibliography}{99}

\bibitem{anaby} M.H. Annaby, A.E. Hamza, K.A. Aldwoah, Hahn difference operator and associated Jackson–N\"{o}rlund integrals, J. Optim. Theory Appl. \textbf{154} (2012), 133--153.

\bibitem{AGMMB-00}I. Area, E. Godoy, F. Marcellán, J.J. Moreno–Balcázar, Inner products involving $q$–differences: the little $q$–Laguerre--Sobolev polynomials, J. Comput. Appl. Math. \textbf{118} (2000), 1--22.

\bibitem{buendia} E. Buendia, J.S. Dehesa, F.J. Galvez, The distribution of the zeros of the polynomial eigenfunctions
of ordinary differential operators of arbitrary order, in
\textit{Orthogonal polynomials and their applications}, M. Alfaro et al.
(eds.), Lecture Notes in Mathematics, vol. \textbf{1329}, Springer--Verlag,
Berlin, 1988, 222--235.


\bibitem{chihara} T.S.~Chihara, {\em An Introduction to Orthogonal Polynomials},
Gordon and  Breach, New York, 1978.


\bibitem{dini} J. Dini, {\em Sur les formes lin\'{e}aires et les polyn\^{o}mes
orthogonaux de Laguerre--Hahn}, Th\`ese de doctorat, Univ. Pierre et Marie Curie, Paris, 1988.

\bibitem{everit} W.N. Everitt,  L.L. Littlejohn,  Orthogonal polynomials and
spectral theory: a survey, in \textit{Orthogonal
Polynomials and their applications}, C. Brezinski, L. Gori and A.
Ronveaux (eds.). IMACS Annals on Computing and
Applied Mathematics, J.C. Baltzer A.G. Basel, 1991, 21--55.

\bibitem{Filipuketal} G. Filipuk, J.F. Mañas–Mañas, J.J.  Moreno–Balcázar, C. Rodríguez–Perales, Second–order difference equation for quasi–orthogonal polynomials related to Hahn difference operator, J. Math. Anal. Appl. \textbf{557} (2026), art. 130268.


\bibitem{foupou-2004} M. Foupouagnigni, W. Koepf, A. Ronveaux,  Factorization of fourth--order differential equations for perturbed classical orthogonal polynomials,
J. Comput. Appl. Math. \textbf{162}(2) (2004),  299--326.

\bibitem{fou-paco} M. Foupouagnigni, F. Marcell\'{a}n,  Characterization of the $D_w$--Laguerre--Hahn functionals, J. Difference Equ. Appl. \textbf{8}(8) (2002), 689--717.

\bibitem{foupou-2001} M. Foupouagnigni, A. Ronveaux, M.N. Hounkonnou,  The fourth--order difference equation satisfied by the associated orthogonal polynomials of the $D_{q}$--Laguerre-Hahn class, J. Differ. Equations Appl. \textbf{7}(3) (2001), 445--472.


\bibitem{ghressi-2009} A. Ghressi, L. Kh\'eriji, The symmetrical $H_{q}$--semiclassical orthogonal polynomials of class one, SIGMA
\textbf{5} (2009), art. 76.


\bibitem{ghressi-LH} A. Ghressi, L. Kh\'eriji, M. Tounsi, An Introduction to the $q$--Laguerre--Hahn Orthogonal $q$--Polynomials,
SIGMA Symmetry Integrability Geom. Methods Appl. \textbf{7} (2011), art. 92.


\bibitem{hahn-q} W. Hahn,  \"{U}ber Orthogonalpolynome, die $q$--Differenzengleichungen gen\"{u}gen, Math. Nachr. \textbf{2} (1949), 4--34.


\bibitem{hahn-ing} W. Hahn, On Differential Equations for Orthogonal Polynomials, Funkcial. Ekvac. \textbf{21}(1) (1978), 1--9.


\bibitem{ismail-book} M.E.H. Ismail, {\em Classical and Quantum Orthogonal Polynomials in One Variable, Encyclopedia of Mathematics and Its Applications}, vol. \textbf{98}, Cambridge University Press, Cambridge, 2005.

\bibitem{khal-zelia} M. Khalfallah, P. Maroni, Z. Rocha, On a general method for deriving a fourth--order differential
equation satisfied by Laguerre--Hahn orthogonal polynomials with new results for the class 0 analogous to Hermite, Numer. Algorithms (2026). \url{https://doi.org/10.1007/s11075-026-02447-z}.


\bibitem{kheriji} L. Kh\'eriji, An introduction to the $H_{q}$--semiclassical orthogonal polynomials, Methods Appl. Anal. \textbf{10}(3) (2003), 387--411.

\bibitem{kheriji-mar} L. Kh\'eriji, P. Maroni,  The $H_{q}$--classical orthogonal polynomials, Acta Appl. Math. \textbf{71}(1) (2002), 49--115.

\bibitem{koek} R. Koekoek, P.A. Lesky, R.F. Swarttouw, {\em Hypergeometric Orthogonal Polynomials and Their $q$--Analogues}, Springer Monographs in Mathematics, Springer, Berlin, 2010.

\bibitem{lag} E. Laguerre, Sur la r\'eduction en fractions continues d'une fonction qui satisfait \`a
une \'equation lin\'eaire du premier ordre \`a coefficients rationnels,   Bull. Soc. Math. France \textbf{8} (1880), 21--27.


\bibitem{magnus-LHsnul} A.P. Magnus, Associated Askey--Wilson polynomials as Laguerre--Hahn orthogonal polynomials,  in
\textit{Orthogonal polynomials and their applications}, M. Alfaro et al.
(eds.), Lecture Notes in Mathematics, vol. \textbf{1329}, Springer--Verlag,
Berlin, 1988, 261--278.

\bibitem{magnus-snul-sc} A.P. Magnus,  Special nonuniform lattice (snul) orthogonal
polynomials on discrete dense sets of points, J. Comput. Appl.
Math. \textbf{65}(1--3) (1995), 253--265.

\bibitem{magnus-jcam} A.P. Magnus, Painlev\'e--type differential equations for the recurrence coefficients of
semi--classical orthogonal polynomials,  J. Comput. Appl. Math. \textbf{57}(1--2) (1995), 215--237.

 \bibitem{maroni} {P. Maroni},  Une th\'eorie alg\'ebrique des  polyn\^omes orthogonaux. Application aux polyn\^omes orthogonaux
 semi-classiques, in \textit{Orthogonal Polynomials and their applications}, C. Brezinski, L. Gori and A. Ronveaux (eds.). IMACS Annals on Computing and
Applied Mathematics, J.C. Baltzer A.G. Basel, 1991, 95--130.


\bibitem{medem} J.C. Medem, R. Álvarez--Nodarse, F. Marcellán, On the $q$–polynomials: a distributional study, J. Comput. Appl. Math. \textbf{135} (2001) 157--196.

\bibitem{mar-mejri} P. Maroni, M. Mejri, The symmetric $D_{\omega}$--semi--classical orthogonal polynomials of class one,
Numer. Algorithms \textbf{49}(1-4) (2008), 251--282.

\bibitem{mejri} M. Mejri,  $q$--extension of some symmetrical and semi--classical orthogonal polynomials of class one, Appl.
Anal. Discrete Math. \textbf{3}(1) (2009), 78--87.


\bibitem{niki-sus-uv} A.F. Nikiforov, S.K. Suslov, V.B. Uvarov, {\em{Classical Orthogonal
Polynomials of a Discrete Variable}}, Springer, Berlin, 1991.

\bibitem{nik-MIR} A.F. Nikiforov, V.B. Uvarov, {\em Special Functions of Mathematical Physics: A unified Introduction with Applications},
Birkh\"{a}user, Basel, Boston, 1988.


\bibitem{ormerod-etal} C.M. Ormerod, N.S. Witte, P. J. Forrester, Connection preserving deformations and $q$--semi--classical orthogonal polynomials, Nonlinearity \textbf{24}(9) (2011), 2405–2434.

 \bibitem{2026-Rebocho} M. N. Rebocho, A note on Laguerre-Hahn orthogonal polynomials: from 1984 to 2024, Numer. Algorithms \textbf{102} (2026), 1661–1675.

\bibitem{smaili} N.E. Smaili, {\em{Les polyn\^omes E-semi-classiques de classe z\'ero}}, Th\`ese de doctorat, Univ. Pierre et Marie Curie, Paris, 1987.

\bibitem{stieltjes} T.J. Stieltjes, Recherches sur les fractions continues, Ann. Fat. Sci. Toulouse \textbf{8} (1894),
51--122; Ann Fat. Sci. Toulouse \textbf{9} (1895) Al--47.

\bibitem{szego} {G. Szeg{\H o}}, {\em Orthogonal Polynomials}, Amer. Math. Soc. Colloq. Publ. \textbf{23},  Amer. Math. Soc. Providence, RI, 1975 (fourth edition).


\bibitem{VA-lecnotes} W. Van Assche, {\em Asymptotics for Orthogonal Polynomials}, Lecture Notes in Mathematics,
vol. \textbf{1265}, Springer--Verlag, Berlin, 1987.


\end{thebibliography}
 \end{document}